\documentclass[12pt]{amsart}

\usepackage{amssymb, mathrsfs}
\usepackage{graphicx}

\usepackage{lineno}

\usepackage{color}

\newtheorem{theorem}{Theorem}
\newtheorem{lemma}[theorem]{Lemma}
\newtheorem{proposition}[theorem]{Proposition}

\newtheorem{corollary}[theorem]{Corollary}

\newtheorem{them}{Theorem}
\newtheorem{lema}[them]{Lemma}

\newtheorem{example}[theorem]{Example}
\newtheorem{remark}[theorem]{Remark}

\begin{document}

\title{Roman domination excellent graphs: trees}

\author[]{Vladimir Samodivkin}
\address{Department of Mathematics, UACEG, Sofia, Bulgaria}
\email{vl.samodivkin@gmail.com}
\today

\begin{abstract}
A Roman dominating function (RDF) on a graph $G = (V, E)$ 
is a labeling $f : V  \rightarrow \{0, 1, 2\}$ such
that every vertex with label $0$ has a neighbor with label $2$. 
The weight of $f$ is the value $f(V) = \Sigma_{v\in V} f(v)$
The Roman domination number, $\gamma_R(G)$, of $G$ is the
minimum weight of an RDF on $G$.
An RDF of minimum weight is called a $\gamma_R$-function.
A graph G is said to be $\gamma_R$-excellent if for each vertex $x \in V$
there is a $\gamma_R$-function $h_x$ on $G$ with $h_x(x) \not = 0$. 
We present  a constructive characterization of $\gamma_R$-excellent trees using labelings. 
A graph $G$ is said to be in class $UVR$ if $\gamma(G-v) = \gamma (G)$ for each $v \in V$, 
where $\gamma(G)$ is the domination number of $G$. 
 We show that each tree in $UVR$ is $\gamma_R$-excellent.  
\end{abstract}

\maketitle


\section{Introduction and preliminaries}

For basic notation and graph theory terminology not explicitly defined here, we
in general follow Haynes, Hedetniemi and Slater \cite{hhs1}. Specifically, let $G$ be a
graph with vertex set $V (G)$ and edge set $E(G)$. 
A {\em spanning subgraph} for $G$ is a subgraph of $G$ which contains every vertex of $G$. 
In a graph $G$, for a subset $S \subseteq V (G)$ the {\em subgraph induced} by $S$ is the graph
$\left\langle S \right\rangle$ with vertex set $S$ and edge set $\{xy \in E(G) \mid x, y \in S\}$.
The complement $\overline{G}$ of $G$ is the graph whose
vertex set is $V (G)$ and whose edges are the pairs of nonadjacent vertices of $G$.
We write $K_n$ for the {\em complete graph} of order $n$ and $P_n$ for the  {\em path} 
on $n$ vertrices. Let $C_m$ denote the {\em cycle} of length $m$.
 For any vertex $x$ of a graph $G$, $N_G(x)$ denotes the set of all neighbors 
of $x$ in $G$, $N_G[x] = N_G(x) \cup \{x\}$  and the degree of $x$ is $deg_G(x) = |N_G(x)|$. 
The {\em minimum} and {\em maximum} degrees
 of a graph $G$ are denoted by $\delta(G)$ and $\Delta(G)$, respectively.
For a subset $S \subseteq V (G)$, let $N_G[S] = \cup_{v \in S}N_G[v]$.
The {\em external private neighborhood} $epn(v, S)$ of $v \in S$ is defined by
$epn(v, S)= \{ u \in V(G)-S \mid N_G(u) \cap S = \{v\}\}$. 
A {\em leaf} is a vertex of degree one and a {\em stem} is a vertex adjacent to a leaf.
If $F$ and $H$ are disjoint graphs,  $v_F \in V(F )$ and $v_H \in V(H)$, 
 then the {\em coalescence}  $(F \cdot H)(v_F , v_H : v)$ of $F$ and $H$ via $v_F$ and $v_H$, 
is the graph obtained from the union of $F$ and $H$ by identifying $v_F$ and $v_H$
in a vertex labeled $v$. If  $F$ and $H$ are graphs with exactly  one vertex in common, say $x$,  
then the {\em coalescence} $(F \cdot H)(x)$                                   
 of $F$ and $H$ via $x$ is the union of $F$ and $H$. 

Let $\mathtt{Y}$ be a finite set of  integers which has  positive as well as non-positive  elements.   
Denote by $P({\mathtt{Y}})$ the collection of all subsets of $\mathtt{Y}$.   
Given  a graph $G$,  for a $\mathtt{Y}$-valued function $f :  V(G) \rightarrow  \mathtt{Y}$
 and a subset $S$ of $V (G)$ we define $f(S) = \Sigma_{v\in S} f(v)$.  
The {\em weight}  of $f$ is $f(V(G))$.  
A $\mathtt{Y}$-{\em valued Roman dominating function} on a graph $G$  
is a function $f :  V(G) \rightarrow  \mathtt{Y}$ satisfying the conditions: 
(a) $f(N_G[v]) \geq 1$ for each $v \in V(G)$, and 
(b)  if $v \in V(G)$ and $f(v) \leq 0$, then there is $u_v \in N_G(v)$ with $f(u_v) = \max \{k \mid k \in \mathtt{Y}\}$. 
For a $\mathtt{Y}$-valued Roman dominating function $f$ on a graph $G$,
 where $\mathtt{Y} = \{r_1, r_2,\dots,r_k\}$ and $r_1 < r_2< \dots < r_k$, 
let $V^f_{r_i} = \{v \in V(G) \mid f(v) = r_i\}$ for $i = 1, ..,k$. 
Since these $k$ sets determine $f$, we can equivalently write $f = (V^f_{r_1}; V^f_{r_2}; \dots ; V^f_{r_k})$.
If $f$ is $\mathtt{Y}$-valued Roman dominating function on a graph $G$
 and $H$ is a subgraph  of $G$, then we denote the restriction of $f$ on $H$ by $f|_H$.
The $\mathtt{Y}$-{\em Roman domination number} of a graph $G$, denoted $\gamma_{{\mathtt{Y}}R}(G)$, 
is defined to be  the minimum weight of a $\mathtt{Y}$-valued dominating function on $G$. 
As examples, let us mention: 
(a) the domination number $\gamma(G) \equiv \gamma_{{\{0,1\}}R}(G)$, 
(b) the minus domination number \cite{dhhm}, where $\mathtt{Y} = \{-1,0,1\}$, 
(c) the signed domination number \cite{dhhs}, where $\mathtt{Y}=\{-1,1\}$, 
(d) the Roman domination number $\gamma_R(G) \equiv \gamma_{{\{0,1,2\}}R}(G)$ \cite{cdhh}, and 
 (e) the signed Roman domination number  \cite{ahlzs}, where  $\mathtt{Y} = \{-1,1,2\}$. 
A $\mathtt{Y}$-valued Roman dominating function $f$ on $G$ 
with  weight $\gamma_{{\mathtt{Y}}R}(G)$ is called 
 a $\gamma_{\mathtt{Y}R}$-{\em function} on $G$.  

Now we  introduce a new partition of a vertex set of a graph, which play a key role  in the paper.
In  determining this  partition, all $\gamma_{\mathtt{Y}R}$-functions of a graph  are necessary. 
For each  $\mathtt{X} \in P({\mathtt{Y}})$ we define the set  $V^{\mathtt{X}}(G) $ 
 as consisting of all $v \in V(G)$ with $\{f(v) \mid f \mbox{\ is a } \gamma_{\mathtt{Y}R}\mbox{-function on } G\} = \mathtt{X}$. 
Then all   members of the family $(V^{\mathtt{X}}(G))_{\mathtt{X} \in P(\mathtt{Y} )}$
 clearly form a partition of $V(G)$. 
We call this partition the $\gamma_{\mathtt{Y}R}$-\emph{partition of }$G$. 

Fricke et al. \cite{fhhhl}  in 2002 began the study of graphs,
which are excellent with respect to various graph parameters.
Let us concentrate here on the parameter $\gamma_{{\mathtt{Y}}R}$.
 A vertex $v \in V(G)$ is said to be  (a) $\gamma_{{\mathtt{Y}}R}$-\emph{good}, 
if $h(v) \geq 1$ for some $\gamma_{{\mathtt{Y}}R}$-function $h$ on $G$, and 
(b) $\gamma_{{\mathtt{Y}}R}$-\emph{bad}  otherwise.   
A  graph $G$ is said to be $\gamma_{\mathtt{Y}R}$-{\em excellent}  if all vertices of $G$
 are $\gamma_{\mathtt{Y}R}$-good.                                       
Any vertex-transitive graph is $\gamma_{\mathtt{Y}R}$-excellent. 
Note also that the set of all $\gamma$-good and the set of all $\gamma$-bad vertices of a graph $G$ 
form the $\gamma$-partition of $G$. For further results on this topic see e.g.   \cite{bs,hh1,hh,h2,j,sam1,sam2}.

In this paper we begin an investigation of  $\gamma_{\mathtt{Y}R}$-excellent graphs in the case when   $\mathtt{Y} = \{0,1,2\}$.
In what follows we shall write $\gamma_R$ instead of $\gamma_{\mathtt{\{0,1,2\}}R}$, 
and we shall abbreviate a $\{0,1,2\}$-valued Roman dominating function to an RD-function.  
Let us describe  all members of  the $\gamma_R$-partition of any graph $G$ 
 (we write $V^i(G)$, $V^{ij}(G)$ and  $V^{ijk}(G)$ instead of $V^{\{i\}}(G), V^{\{i,j\}}(G)$ and $V^{\{i, j, k\}}(G)$, respectively).
\begin{itemize}
\item[(i)]   $V^i(G) = \{x \in V(G) \mid f(x) = i \mbox{ for each } \gamma_R\mbox{-function } f \mbox{ on } G\}, i = 1,2,3$;
\item[(ii)]  $V^{012}(G) = \{x \in V(G) \mid \mbox{there are } \gamma_R\mbox{-functions } f_x,g_x,h_x  \mbox{ on } G \mbox{ with }$

             \hspace{1.5cm}                                           $f_x(x) =0, g_x(x) = 1 \mbox{ and }  h_x(x)=2\}$;
\item[(iii)]  $V^{ij}(G) = \{x \in V(G) - V^{012}(G) \mid \mbox{there are } \gamma_R\mbox{-functions } f_x \mbox{ and } g_x  \mbox{ on } G $

                     \hspace{1.4cm}    $\mbox{ with }  f_x(x) =i \mbox{ and } g_x(x) = j \} , 0\leq i < j \leq 2$. 
\end{itemize}

Clearly a  graph $G$ is $\gamma_R$-excellent if and only if $V^0(G) = \emptyset$.

It is often of interest to known how the value of a graph parameter is affected 
when a small change is made in a graph.
 In this connection, Hansberg, Jafari Rad and Volkmann studied in \cite{rhv}
changing and unchanging of the Roman domination number 
 of a graph when a vertex is deleted, or an edge is added. 

\begin{lema}\label{minus} {\rm(\cite{rhv})} 
Let $v$ be a vertex of a graph $G$. Then $\gamma_R(G-v) < \gamma_R(G)$ 
 if and only if there is a $\gamma_R$-function $f = (V_0^f; V_1^f; V_2^f)$ on $G$
 such that $v \in V_1^f$. If $\gamma_R(G-v) < \gamma_R(G)$ then $\gamma_R(G-v) = \gamma_R(G) - 1$. 
\end{lema}

 Lemma \ref{minus} implies that $V^1(G), V^{01}(G), V^{12}(G), V^{012}(G)$
 form a partition of $V^-(G) = \{x \in V(G) \mid \gamma_R(G-x) +1 =  \gamma(G)\}$.

\begin{lema}\label{addedge} {\rm(\cite{rhv})} 
Let $x$ and $y$ be  non-adjacent vertices  of a graph $G$.
Then $\gamma_R(G) \geq \gamma_R(G+xy) \geq \gamma_R(G)-1$. 
Moreover, $\gamma_R(G+xy) = \gamma_R(G)-1$ if and only if  there is a $\gamma_R$-function 
$f$ on $G$  such that $\{f(x), f(y)\} = \{1, 2\}$.
\end{lema}

 The same authors defined  the following two classes of graphs: 
\begin{itemize}
\item[(i)]  $\mathcal{R}_{CVR}$  is the class of graphs $G$ 
                     such that $\gamma_R(G-v) < \gamma_R(G)$ for all $v \in V (G)$.
\item[(ii)]  $\mathcal{R}_{CEA}$  is the class of graphs $G$ 
                     such that $\gamma_R(G+e) < \gamma_R(G)$ for all $e\in E (\overline{G})$.
\end{itemize}

\begin{remark} \label{extension} 
By the above two lemmas it easy follows that:
\begin{itemize}
 \item[(i)] each graph in  $\mathcal{R}_{CVR} \cup \mathcal{R}_{CEA}$ is $\gamma_R$-excellent, 
\item[(ii)] if $G$ is a $\gamma_R$-excellent graph, $e \in E(\overline{G})$ and $\gamma_R(G) = \gamma_R(G+e)$, 
                    then $G+e$ is  $\gamma_R$-excellent,  
\item[(iii)]  each graph (in particular each $\gamma_R$-excellent graph)
                     is a spanning subgraph of a graph in $\mathcal{R}_{CEA}$ with the same Roman domination number.    
\end{itemize} 
\end{remark}


Denote by $\mathtt{G}_{n,k}$ the family of all  mutually non-isomorphic $n$-order $\gamma_R$-excellent connected 
graphs having the Roman domination number equal to $k$. 
With the family $\mathtt{G}_{n,k}$,  we associate  the poset  $\mathbb{RE}_{n,k} = (\mathtt{G}_{n,k}, \prec)$
with the order $\prec$ given by $H_1 \prec H_2$ if and only if
$H_2$ has a spanning subgraph which is isomorphic to $H_1$ 
 (see \cite{tro} for terminology on posets).
Remark \ref{extension} shows that all maximal elements of  $\mathbb{RE}_{n,k}$ are in $\mathcal{R}_{CEA}$.
Here we concentrate on the set of all minimal elements of $\mathbb{RE}_{n,k}$. 
Clearly a  graph $H \in \mathtt{G}_{n,k}$ is a minimal element of $\mathbb{RE}_{n,k}$
if and only if for each $e \in E(H)$ at least  one of the following holds: 
(a) $H-e$ is not connected,  (b) $\gamma_R(H) \not = \gamma_R(H-e)$, and  
(c) $H-e$ is not $\gamma_R$-excellent.  
All trees in $\mathtt{G}_{n,k}$ are obviously minimal elements of $\mathbb{RE}_{n,k}$.

The remainder of this paper is organized as follows.
In Section 2, we formulate our main result, namely, 
 a constructive characterization of $\gamma_R$-excellent trees. 
We present a proof of this result in Sections 3 and 4. 
Applications of our main result are given in Sections 5 and 6.  
We conclude in Section 7 with some open problems.

We end this section with the following useful result. 
\begin{lema}\label{on} {(\rm{\cite{cdhh}})} 
Let $f = (V_0^f; V_1^f; V_2^f)$ be any $\gamma_R$-function on a graph $G$. 
Then each component of a  graph $\left\langle V_1^f \right\rangle$ has order at most $2$  and 
no edge of $G$ join $V_1^f$ and $V_2^f$. 
\end{lema}
In most cases Lemma \ref{minus},  Lemma \ref{addedge} and Lemma \ref{on} 
                                   will be used in the sequel without specific reference.


\section{The main result}

In this section we present  a constructive characterization of $\gamma_R$-excellent trees using labelings. 
We define a 
{\em labeling}  of a tree $T$  as a function $S:V (T ) \rightarrow\{A,B,C,D\}$. 
A labeled tree is denoted by a pair $(T , S)$. 
 The label of a vertex $v$ is also called its {\em status},  denoted $sta_T(v:S)$
 or $sta_T(v)$ if the labeling $S$ is clear from context. 
 We denote the sets of vertices of status $A, B,C$ and $D$ 
 by $S_A(T ), S_B(T), S_C(T)$ and $S_D(T )$, respectively. 
	In all figures in this paper we use $\bullet$ for a vertex of status $A$, 
 {\tiny  $\blacktriangledown$} for a vertex of status $B$, {\tiny $\blacklozenge$}  for a vertex of status $C$,    
	and  $\circ$ for a vertex of status $D$.
	If $H$ is a subgraph  of $T$, then we denote the restriction of $S$ on $H$ by $S|_H$.

\begin{figure}[h]
	\centering
		\includegraphics{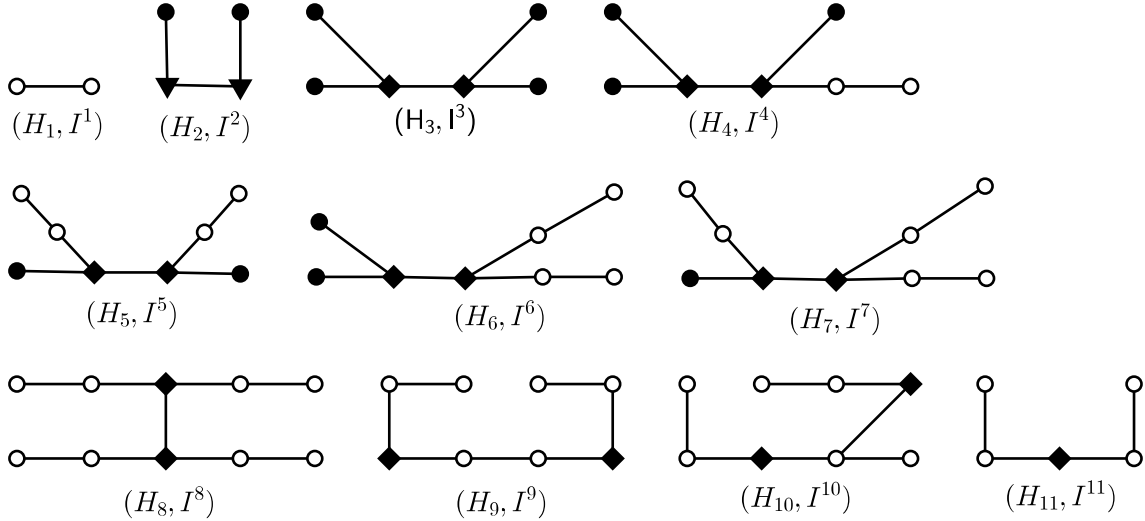}
	\caption{All trees with $|L_B \cup L_C| \leq 2$. }
	\label{fig:Fig 1}
\end{figure}

\begin{figure}[htbp]
	\centering
		\includegraphics{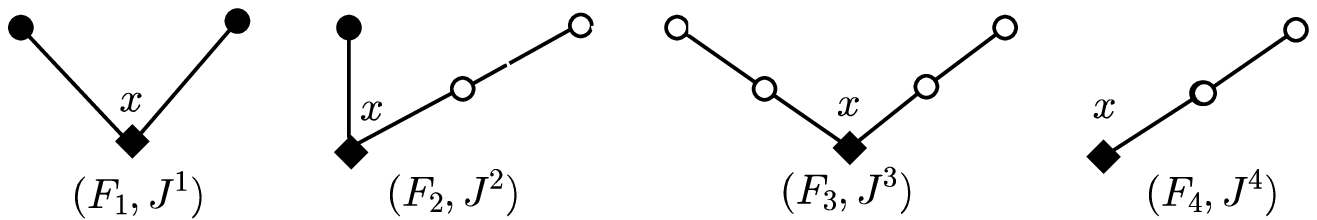}
	\caption{}
	\label{fig:F1234}
\end{figure}

	To state a characterization of $\gamma_R$-excellent trees, we introduce  four  types of operations. 
Let $\mathscr{T}$ be the family of labeled trees $(T, S)$ that can be obtained from a sequence of labeled trees 
$\tau: (T^1, S^1), \dots, (T^j, S^j)$, ($j \geq 1$),  
such that  $(T^1, S^1)$ is in $\{(H_1, I^1),.., (H_{5}, I^5)\}$ (see  Figure \ref{fig:Fig 1}) and $(T, S) = (T^j, S^j)$, 
and, if $j \geq 2$,  $(T^{i+1}, S^{i+1})$ can be obtained recursively from $(T^i, S^i)$ 
 by one of the  operations $O_1,  O_2, O_3$ and $O_4$ listed below;
 in this case  $\tau$ is said to be  a  $\mathscr{T}$-{\em sequence}   of $T$. 
When the context is clear we shall write $T \in \mathscr{T}$ instead of $(T,S)  \in \mathscr{T}$.
\medskip
 \\
{\bf Operation}  {$O_1$}.  The labeled tree $(T^{i+1},S^{i+1})$ is obtained from 
$(T^i,S^i)$ and $(F,J) \in \{(F_1, J^1), (F_2, J^2), (F_3, J^3)\}$ (see  Figure \ref{fig:F1234})  by adding
 the edge $ux$, where $u \in V (T_i)$, $x \in V(F)$ and  $sta_{T^i}(u) = sta_F(x) = C$.
\medskip
 \\
{\bf Operation}  {$O_2$}.  The labeled tree $(T^{i+1},S^{i+1})$ is obtained from 
 $(T^i,S^i)$ and $(F_4, J^4)$ (see  Figure \ref{fig:F1234})  by adding  the edge $ux$, 
where $u \in V (T^i)$, $x \in V(F_4)$,  $sta_{T^i}(u) = D$, and $sta_{F_4}(x) = C$.
\medskip
 \\
{\bf Operation}   {$O_3$}.   The labeled tree $(T^{i+1},S^{i+1})$ is obtained from 
 $(T^i,S^i)$ and $(H_k, I^k)$, $k \in \{2,..,7\}$ (see  Figure \ref{fig:Fig 1}),   
in such a way that  $T^{i+1}= (T^i \cdot H_k)(u,v:u)$, 
where $sta_{T^i}(u) = sta_{H_k}(v) = A$, and $sta_{T^{i+1}}(u) = A$.
\medskip
 \\
{\bf Operation}  {$O_4$}. 
The labeled tree $(T^{i+1},S^{i+1})$ is obtained from 
 $(T^i,S^i)$ and $(H_k, I^k)$, $k \in \{3,4,6\}$ (see  Figure \ref{fig:Fig 1}),   
in such a way that  $T^{i+1}= (T^i \cdot H_k)(u,v:u)$, 
where $sta_{T^i}(u) = D$, $sta_{H_k}(v) = A$, and $sta_{T^{i+1}}(u) = D$.
\medskip

Remark that if $y \in V(T^i)$ and $i \leq k \leq j$, then $sta_{T^i}(y) = sta_{T^k}(y)$.  
Now we are prepared to state the main result. 

\begin{theorem}\label{main}
Let $T$ be a  tree of order at least $2$. 
Then $T$ is $\gamma_R$-excellent if and only if 
there is a labeling $S:V (T ) \rightarrow\{A,B,C,D\}$ 
such that $(T, S)$ is in $\mathscr{T}$. 
Moreover, if $(T, S) \in \mathscr{T}$ then 
	\begin{itemize}
	\item[($\mathcal{P}_1$)] 	 $S_B(T) = \{x \in V^{02}(T) \mid deg(x) =2 $ and $|N(x) \cap V^{02}(T)| =1\}$, 
	$S_A(T) = V^{01}(T)$, $S_D(T) = V^{012}(T)$, and $S_C(T) = V^{02}(T) - S_B(T)$.  		
\end{itemize}
\end{theorem}

\section{Preparation for the proof of Theorem \ref{main}}

\subsection{\bf Coalescence}

We shall concentrate on the coalescence of two graphs via a vertex in $V^{01}$
 and derive the properties which will be needed for the proof of our main result.

\begin{proposition} \label{coal01}
Let $G = (G_1 \cdot G_2)(x)$ be a connected graph and $x \in V^{01}(G)$.
 Then the following holds.
\begin{itemize}
\item[(i)] If $f$ is a $\gamma_R$-function on $G$ and $f(x)=1$,  
           then $f|_{G_i}$ is a $\gamma_R$-function  on $G_i$, 
           and $f|_{G_i-x}$ is a $\gamma_R$-function on $G_i-x$, $i=1,2$.   
\item[(ii)] $\gamma_R(G) = \gamma_R(G_1) + \gamma_R(G_2) -1$.
\item[(iii)] If $h$ is a $\gamma_R$-function  on $G$ and $h(x)=0$, then 
             exactly one of the following holds:
        \begin{itemize}
        \item[(iii.1)] $h|_{G_1}$ is a $\gamma_R$-function  on $G_1$, 
                       $h|_{G_2-x}$ is a $\gamma_R$-function  on $G_2-x$, and 
                       $h|_{G_2}$ is no  RD-function on $G_2$;
        \item[(iii.2)] $h|_{G_1-x}$ is a $\gamma_R$-function  on $G_1-x$, 
                       $h|_{G_1}$  is no RD-function on $G_1$, and 
                       $h|_{G_2}$ is a $\gamma_R$-function  on $G_2$. 
        \end{itemize}  
\item[(iv)] Either $\{x\} = V^{01}(G_1) \cap V^{01}(G_2)$ or  
            $\{x\} = V^{01}(G_i) \cap V^{1}(G_j)$, where $\{i,j\} = \{1,2\}$.       
\end{itemize}
\end{proposition}
\begin{proof}[\bf Proof.]
(i) and (ii): Since $f(x) =1$,  $f|_{G_i}$ is an RD-function  on $G_i$,  and $f|_{G_i-x}$ is an RD-function on $G_i-x$, $i=1,2$. 
      Assume $g_1$ is a $\gamma_R$-function on $G_1$ with $g(V(G_1)) < f|_{G_1}(V(G_1))$. 
      Define an RD-function $f^\prime$ as follows: $f^\prime(u) = g_1(u)$ for all $u \in V(G_1)$ and 
      $f^\prime(u) = f(u)$ when $u \in V(G_2-x)$. 
      Then $f^\prime(V(G)) = g_1(V(G_1)) + f|_{G_2-x}(V(G_2-x)) < f(V(G))$, a contradiction.
      Thus, $f|_{G_i}$ is a $\gamma_R$-function on $G_i$, $i=1,2$. 
			Now by Lemma \ref{minus}, $f|_{G_i-x}$ is a $\gamma_R$-function on $G_i-x$, $i=1,2$.
			Hence 
      $\gamma_R(G) = f|_{G_1}(V(G_1)) + f|_{G_2}(V(G_2)) - f(x) = \gamma_R(G_1) + \gamma_R(G_2) -1$.
      
 (iii) First note that $h(x) = 0$ implies $h|_{G_i}$ is an RD-function on $G_i$ for some $i\in \{1,2\}$, say $i=1$. 
       If $h|_{G_2}$ is an RD-function on $G_2$ then $\gamma_R(G) =h(V(G)) \geq \gamma_R(G_1) + 
        \gamma_R(G_2)$, a contradiction with (ii). 
        Thus, $h|_{{G_2}-x}$ is an RD-function on $G_2-x$. 
        Now we have 
        $\gamma_R(G_1) + \gamma_R(G_2) -1 = \gamma_R(G) = h(V(G)) = 
         h|_{G_1}(V(G_1)) + h|_{G_2-x}(V(G_2-x)) \geq \gamma_R(G_1) + (\gamma_R(G_2) -1)$. 
         Hence $h|_{G_1}$ is a $\gamma_R$-function on $G_1$ and $h|_{G_2-x}$ 
         is a   $\gamma_R$-function on $G_2-x$. 
         
 (iv)  Let $f_1$ be a $\gamma_R$-function on $G_1$. Assume first that $f_1(x) = 2$. 
      Define an RD-function $g$ on $G$ as follows: 
      $g(u) = f_1(u)$ when $u \in V(G_1)$ and $g(u) = f(u)$ when 
      $u \in V(G_2-x)$, where $f$ is defined  as in (i). 
			The weight of $g$ is $\gamma_R(G_1) + (\gamma_R(G_2)+1) - 2= \gamma_R(G)$. 
      But $g(x) =2$ and $x \in V^{01}(G)$, a contradiction. 
      Thus $f_1(x) \not = 2$.
      Now by (i) we have  $x \in V^1(G_i) \cup V^{01}(G_i)$, $i=1,2$, 
			and  by (iii), $x \in V^{01}(G_j)$ for some $j \in \{1,2\}$.   
\end{proof}

\begin{proposition} \label{coales12}
Let $G = (G_1 \cdot G_2)(x)$, where $G_1$ and $G_2$ are connected graphs and $\{x\} = V^{01}(G_1)\cap V^{01}(G_2)$. 
\begin{itemize}
\item[(i)] If $f_i$ is a $\gamma_R$-function on $G_i$ with $f_i(x)=1$, $i=1,2$, 
           then the function $f: V(G) \rightarrow \{0,1,2\}$ with 
           $f|_{G_i} = f_i$, $i=1,2$,  is a $\gamma_R$-function  on $G$.    
\item[(ii)] $\gamma_R(G) = \gamma_R(G_1) + \gamma_R(G_2) -1$.
\item[(iii)]  Let $V_R =  \{V^0, V^1, V^2, V^{01}, V^{02}, V^{12}, V^{012}\}$. 
                     Then for any $A \in V_R$, $A(G_1) \cup A(G_2) = A(G)$. 
\end{itemize}
\end{proposition}
\begin{proof}[\bf Proof.]
(i) and (ii): Note that $f$ is an RD-function on $G$ and 
$\gamma_R(G) \leq f(V(G)) = f_1(V(G_1)) + f_2(V(G_2)) - f(x) = \gamma_R(G_1) + \gamma_R(G_2) -1$.
Now let $h$ be any $\gamma_R$-function on $G$. 

{\it Case} 1: $h(x) \geq 1$. 
              Then $h|_{G_i}$ is an RD-function on $G_i$, $i=1,2$. 
              If $h(x) = 2$ then $h|_{G_i}$ is no  $\gamma_R$-function on $G_i$, $i=1,2$. 
              Hence $\gamma_R(G) \geq (\gamma_R(G_1) + 1) + (\gamma_R(G_2) + 1) - h(x) = 
              \gamma_R(G_1) + \gamma_R(G_2)$, a contradiction. 
              If $h(x)=1$ then $\gamma_R(G) = h(V(G)) = h(V(G_1)) + h(V(G_2)) - h(x) \geq \gamma_R(G_1) + \gamma_R(G_2) -1$. 
              Thus $h(x) = 1$, $\gamma_R(G) = \gamma_R(G_1) + \gamma_R(G_2) -1$ and 
              $f$ is a $\gamma_R$-function on $G$.
              
{\it Case} 2: $h(x)=0$.
              Then at least one of $h|_{G_1}$ and $h|_{G_2}$ is an RD-function, say the first. 
              If $h|_{G_2}$ is an RD-function on $G_2$ then $h(V(G)) \geq \gamma_R(G_1) + \gamma_R(G_2)$, 
              a contradiction. Hence $h|_{G_2-x}$ is a $\gamma_R$-function on $G_2-x$. 
              But then $\gamma_R(G) = h(V(G)) \geq \gamma_R(G_1) + \gamma_R(G_2-x) \geq 
              \gamma_R(G_1) + \gamma_R(G_2) - 1 \geq \gamma_R(G)$. 
							
							Thus, (i) and (ii) hold. 
\medskip
															
(iii): Let $g_1$ be a $\gamma_R$-function on $G_1$ with $g_1(x)=0$, and $g_2$ a 
       $\gamma_R$-function on $G_2-x$. Then the RD-function $g$ on $G$ for which 
       $g|_{G_1} = g_1$ and $g|_{G_2-x} = g_2$ has weight 
       $g_1(V(G_1)) + g_2(V(G_2-x)) = \gamma_R(G_1) + \gamma_R(G_2-x)  =
        \gamma_R(G_1) + \gamma_R(G_2)-1 = \gamma_R(G)$.  
        Hence by (i), $x \in V^{01}(G) \cup V^{012}(G)$. 
				However, by Case 1 it follows that $h(x) \not = 2$ for any  $\gamma_R$-function $h$ on $G$. 
					Thus $x \in V^{01}(G) $. 
				
				Let $y \in V(G_1-x)$, $l_1$ a $\gamma_R$-function on $G_1$,  and  $h$ a  $\gamma_R$-function on $G$. 
				We shall prove that the following holds.
				\begin{enumerate}
				\item[(*)]  There are a $\gamma_R$-function $l$ on $G$,  and  
				                       a  $\gamma_R$-function $h_1$ on $G_1$ such that $l(y) = l_1(y)$ and $h_1(y) = h(y)$.
				\end{enumerate}
				
				Define an RD-function $l$ on $G$ as $l|_{G_1} = l_1$ and $l|_{G_2-x} = l_2$, where
				$l_2$ is a $\gamma_R$-function on $G_2-x$.  Since   $l(V(G)) = \gamma_R(G_1) + \gamma_R(G_2-x)  = \gamma_R(G)$, 
				$l$ is a $\gamma_R$-function on $G$ and $l(y) = l_1(y)$.

			Assume now that  there is no $\gamma_R$-function $h_1$ on $G_1$ with $h_1(y) = h(y)$. 
			By Proposition \ref{coal01},  $h|_{G_1-x}$ is a $\gamma_R$-function on $G_1 - x$. 
			But then the function $h^\prime: V(G_1) \rightarrow \{0,1,2\}$ defined as 
			$h^\prime(u) = 1$ when $u=x$ 	 and 		$h^\prime (u) = h|_{G_1}(u)$ otherwise, 
			is a  $\gamma_R$-function on $G_1$ with $h^\prime (y) = h|_{G_1}(y)$, a contradiction. 
			
			By (*) and since $x \in V^{01}(G)$, $A(G_1)  =  A(G) \cap V(G_1)$ for any $A \in V_R$. 
			By symmetry,  $A(G_2)  =  A(G) \cap V(G_2)$.
      Therefore  $A(G_1) \cup A(G_2) = A(G)$  for any $A \in V_R$.
		\end{proof}

     \begin{lemma} \label{01012}
		Let $G = (G_1 \cdot G_2)(x)$, where $G_1$ and $G_2$ are connected graphs and $\{x\} = V^{012}(G_1)\cap V^{01}(G_2)$. 
		Then $\gamma_R(G) = \gamma_R(G_1) + \gamma_R(G_2) -1$ and $x \in V^{012}(G)$. 
		\end{lemma}
   \begin{proof}[\bf Proof.]
	Let $f_i$ be a $\gamma_R$-function on $G_i$ with $f_i(x) = 1$, $i = 1,2$. 
	Then the function $f$ defined as $f|_{G_i} = f_i$ is an RD-function on $G_i$, $i = 1,2$.
	Hence $\gamma_R(G) \leq f(V(G)) = \gamma_R(G_1) + \gamma_R(G_2) -1$. 
	Let now $h$ be any  $\gamma_R$-function on $G$. 

	{\bf Case 1}:  $h(x) = 2$. 
	Then $h|_{G_1}$ is a $\gamma_R$-function on $G_1$ and $h|_{G_2}$ is an RD-function on $G_2$ 
	of weight more than $\gamma_R(G_2)$. Hence
	 $\gamma_R(G) =  h(V(G)) \geq \gamma_R(G_1) + (\gamma_R(G_2) + 1) - h(u)$. 
	 Thus $\gamma_R(G) = \gamma_R(G_1) + \gamma_R(G_2) -1$. 
	\medskip
	
	{\bf Case 2}:  $h(x) = 1$. 
			Then obviously $h|_{G_1}$ and $h|_{G_2}$ are $\gamma_R$-functions. 
			Hence  $\gamma_R(G) = \gamma_R(G_1) + \gamma_R(G_2) -1$. 
								
	{\bf Case 3}:  $h(x) = 0$. 
		Hence at least one of $h|_{G_1}$ and $h|_{G_2}$ is a $\gamma_R$-function. 
		If both  $h|_{G_1}$ and $h|_{G_2}$ are $\gamma_R$-functions, then 
		$\gamma_R (G) = \gamma_R(G_1) + \gamma_R(G_2)$, a contradiction. 
		Hence either $h|_{G_1}$ and $h|_{G_2-x}$ are $\gamma_R$-functions,  
		or $h|_{G_1-x}$ and $h|_{G_2}$ are $\gamma_R$-functions.  
   Since  $\{x\} = V^{012}(G_1)\cap V^{01}(G_2)$,  in both cases we have 
	$\gamma_R(G) = \gamma_R(G_1) + \gamma_R(G_2) -1$. 
	
Thus, $\gamma_R(G) = \gamma_R(G_1) + \gamma_R(G_2) -1$ and 	 $x \in V^{012}(G)$. 
	\end{proof}

\subsection{\bf Three lemmas for trees}
\begin{lemma}\label{12}
Let $T$ be a $\gamma_R$-excellent tree of order at least $2$.  
Then $V(T)  = V^{01}(T) \cup V^{012}(T) \cup V^{02}(T)$.
\end{lemma}
\begin{proof}[\bf Proof.]
Let $x \in V(T)$, $y \in N(x)$ and $f$ a $\gamma_R$-function on $T$. 
Suppose $x \in V^1(T)$.  If $f(y) = 1$,  then the RD-function $g$ on $T$ defined as 
$g(x) = 2$, $g(y) = 0$ and $g(u) = f(u)$ for all $u \in V(T) - \{x,y\}$ 
 is a  $\gamma_R$-function on $T$, a contradiction. 
But then $N(x) \subseteq V^0(T)$,  which is impossible.

Suppose now $x \in V^2(T) \cup V^{12}(T)$. Hence $x$ is not a leaf. 
Choose a $\gamma_R$-function $h$ on $T$  such that (a) $h(x) = 2$, and  
(b) $k = |epn[x, V_2^h]|$ to be as small as possible. 
Let  $epn[x, V_2^h] = \{y_1,..,y_k\}$ and denote by $T_i$
 the connected component of $T-x$, which contains $y_i$. 
Hence $h(y_i) =0$ for all $i \leq k$. 
Since $T$ is $\gamma_R$-excellent, there is a $\gamma_R$-function $f_k$ on $T$ with 
$f_k(y_k) \not=0$.  Since $x \in V^2(T) \cup V^{12}(T)$, $f_k(x) \not= 0$. 
If $f_k(y_k) =1$ then $f_k(x) = 1$, which easily implies $x \in V^{012}(T)$, a contradiction. 
Hence $f_k(y_k) = f_k(x) = 2$.  Define a $\gamma_R$-function $l$ on $T$ as 
$l|_{T_k} = f_k|_{T_k}$ and $l(u) = h(u)$ for all $u \in V(T) - V(T_k)$. 
But  $|epn[x, V_2^l]| <k$, a contradiction with the choice of $h$.  
Thus $V^1(T) \cup V^2(T) \cup V^{12}(T)$ is empty, and the required follows. 
\end{proof}

Lemma \ref{12} will be used without specific references.

\begin{lemma}\label{no3}
Let $T$ be a tree and $V^-(T)$ is not empty.   
Then each component of $\left\langle V^-(T)\right\rangle$ is either $K_1$ or $K_2$.  
\end{lemma}
\begin{proof}[\bf Proof.]
Assume that $P: x_1,x_2,x_3$ is a path in $T$ and $x_1,x_2,x_3 \in  V^-(T)$. 
Then there is a $\gamma_R$-function $f_i$ on $T$ with $f_i(x_i) = 1$, $i = 1,2,3$. 
Denote by $T_j$ the connected component of $T-x_2x_j$ that contains $x_j$, $j=1,3$.  
Then $f_2|_{T_j}$ and  $f_j|_{T_j}$ are $\gamma_R$-functions on $T_j$, $j=1,3$.  
Now define a $\gamma_R$-function $h$ on $T$ such that 
$h|_{T_j} = f_j|_{T_j}$, $j=1,3$, and $h(u) = f_2(u)$ when $u \in V(T) - (V(T_1) \cup V(T_3))$. 
But $h(x_1) = h(x_2) = h(x_3) = 1$, a contradiction.
\end{proof}

\begin{lemma}\label{adj}
Let $T$ be a $\gamma_R$-excellent tree of order at least $2$.  
\begin{itemize}
\item[(i)] If $x \in V^{012}(T)$, then $x$ is adjacent to exactly one 
							  	vertex in $V^-(T)$, say $y_1$, and $y_1 \in V^{012}(T)$. 
\item[(ii)] Let $x \in V^{02}(T)$.  If $deg(x) \geq 3$ then $x$ has exactly $2$ neighbors in $V^-(T)$.
						     If $deg(x) = 2$ then  either $N_T(x) \subseteq V^{012}(T)$ or 
								there is a path $u,x,y,z$ in $T$ such that $u,z \in V^{01}(T)$,  $y \in V^{02}(T)$ and $deg(y) = 2$.
	\item[(iii)]  $V^{01}(T)$ is either empty or independent. 
\end{itemize}
\end{lemma}
\begin{proof}[\bf Proof.]
  Let $x \in V^{012}(T) \cup V^{02}(T)$ and  $N(x)  = \{y_1,..,y_r\}$.  
	If $x$ is a leaf, then clearly $x,y_1 \in V^{012}(T)$. So, let  $r \geq 2$. 
								Denote by $T_i$ the connected component of $T-x$ which contains $y_i$, $i \geq 1$. 
								Choose a $\gamma_R$-function $h$ on $T$ such that (a) $h(x) = 2$, and  
								(b) $k = |epn[x, V_2^h]|$ to be as small as possible. Let without loss of generality 
								$epn[x, V_2^h] = \{y_1,..,y_k\}$. By the definition of $h$  it immediately follows that  
								(c) $h|_{T_j}$ is a $\gamma_R$-function on $T_j$ for all $j  \geq k+1$, 
								(d) for each $i \in \{1,..,k\}$, $h|_{T_i}$ is no RD-function on $T_i$, and 
								(e) $h|_{T_i - y_i}$ is a $\gamma_R$-function on $T_i-y_i$, $i \in \{1,..,k\}$. 
								Hence $\gamma_R(T_i) \leq \gamma_R(T_i-y_i) + 1$ for all $i \in \{1,..,k\}$. 
								Assume that  the equality does not hold for some $i \leq k$.  
								Define  an RD-function  $h_i$ on $T$  as follows: 
								$h_i(u) = h(u)$ when $u \in V(T) - V(T_i)$ and $h_i|_{T_i} = h^{\prime}_i$, 
								where $h^{\prime}_i$ is some $\gamma_R$-function on $T_i$. 
								But then either $h_i$ has weight less than $\gamma_R(T)$ or 
								$h_i$ is a $\gamma_R$-function on $T$ with $epn[x, V_2^{h_i}] = epn[x, V_2^{h}]- \{y_i\}$. 
								In both cases we have a contradiction.  
								Thus $\gamma_R(T_i) = \gamma_R(T_i-y_i) + 1$ for all $i \in \{1,..,k\}$. 
								Therefore  $\gamma_R(T) = h(V(T)) = 2 + \Sigma_{i=1}^k( \gamma_R(T_i) - 1) + 
								 \Sigma_{j=k+1}^r \gamma_R(T_j) = 2-k + \Sigma_{i=1}^r \gamma_R(T_i) = 2-k + \gamma_R(T-x)$. 
									Thus  $\gamma_R(T) = 2-k + \gamma_R(T-x)$. 
								\medskip
								
(i)       		Since  $\gamma_R(T-x) + 1 = \gamma_R(T)$, $k=1$.  
							We already know that $h|_{T_j}$ is a $\gamma_R$-function on $T_j$, $j \geq 2$. 
							Assume that $y_j \in V^{012}(T) \cup V^{01}(T)$ for some $j \geq 2$. 
							Then there is a $\gamma_R$-function $l$ on $T$ with $l(y_j) = 1$. 
							Clearly $l|_{T_j}$ is a $\gamma_R$-function  on $T_j$. 
							Now define a $\gamma_R$-function $h^{\prime\prime}$ on $T$ as follows: 
							$h^{\prime\prime}(u) = h(u)$ when $u \in V(T) - V(T_j)$ and $h^{\prime\prime}|_{T_j}  = l|_{T_j}$. 
							But then $h^{\prime\prime}(x) = 2, h^{\prime\prime}(y_j) = 1$ and $xy_j \in E(G)$, which is impossible. 
							Thus, $y_2,..,y_r \in V^{02}(T)$.							
							Define now  $\gamma_R$-functions $h_1$ and $h_2$ on $T$ as follows:  
							$h_1(u) = h_2(u) = h(u)$ for all $u \in V(T) - \{x,y_1\}$,   $h_1 (x) = h_1(y_1) = 1$, 
							$h_2(x) = 0$ and $h_2(y_1) = 2$. 	Thus $y_1 \in V^{012}(T)$.							
\medskip
								
	(ii) 					Since  $\gamma_R(T-x) = \gamma_R(T)$, $k=2$.  
									Recall that $h|_{T_j}$ is a $\gamma_R$-function on $T_j$, $j \geq 3$, and 
									$\gamma_R(T_i-y_i) = \gamma_R(T_i) -1$ for $i = 1,2$.  
										Hence there is  a $\gamma_R$-function $f_i$ on $T_i$ with $f_i(y_i) = 1$, $i = 1,2$. 
									
									Suppose first that $r \geq 3$. As in the proof of (i), we obtain $y_3,..,y_r \in V^{02}(T)$. 
							Hence there is a $\gamma_R$-function $g$ on $T$ such that $g(y_3) = 2$. 
									By the choice of $h$, $g(x) = 0$. 
									Then $g|_{T_i}$ is a $\gamma_R$-function on $T_i$, $i=1,2$. 
									Define now a $\gamma_R$-function $g^\prime$ on $T$ as $g^\prime|_{T_i} = f_i$, $i=1,2$, and 
									$g^\prime(u) = g(u)$ when $u \in V(T) - (V(T_1) \cup V(T_2))$. 
									Since $g^\prime(y_1) = g^\prime(y_2) = 1$, $y_1, y_2 \in V^-(T)$. 
																		
								  So, let $r =2$	and  let $f$ be a $\gamma_R$-function on $T$ with $f(x) = 0$.
									Then there is $y_s$ such that $f(y_s) = 2$, say $s=2$. 
									Hence $y_2 \in V^{02}(T) \cup V^{012}(T)$ and $f|_{T_1}$ is a $\gamma_R$-function on $T_1$. 
									Define the $\gamma_R$-function $l$ on $T$ as $l|_{T_1} = f_1$ and 
									$l(u) = f(u)$ when $u \in V(T) - V(T_1)$. Since $l(y_1) = 1$, $y_1 \in V^{01}(T) \cup V^{012}(T)$. 
									
									Assume first that $y_1 \in V^{012}(T)$.  Then there is a $\gamma_R$-function $f^\prime$ on $T$ 
									with $f^\prime(y_1) = 2$. Since $x \in V^{02}(T)$ and $deg(x) =2$, $f^\prime(x) = 0$. 
									Hence $f^\prime|_{T_2}$ is a  a $\gamma_R$-function on $T_2$. 
									But then we can choose  $f^\prime$ so that  $f^\prime|_{T_2} = f_2$. 
									Thus $y_2 \in V^{012}(T)$. 
									
									So let $y_1 \in V^{01}(T)$ and suppose $y_2 \in V^{012}(T)$.  
									Then there is a $\gamma_R$-function $f^{\prime\prime}$ on $T$ with 
									$f^{\prime\prime}(y_2) = 1$. Since $x \in V^{02}(T)$, $f^{\prime\prime}(x)  =0$ 
									and $f^{\prime\prime}(y_1) = 2$, a contradiction. 
									Thus, if $y_1 \in V^{01}(T)$ then $y_2 \in V^{02}(T)$. 
																							
										Finally, let us consider a path $y_1, x, y_2, z$ in $T$, where 
										$y_1 \in V^{01}(T)$, $x,y_2 \in V^{02}(T)$ and $deg(x) = 2$. 
										Assume to the contrary that $N(y_2) = \{z_1,..,z_s=x\}$ with $s \geq 3$. 
										Denote by $T_{z_p}$ the connected component of $T-y_2$ that contains $z_p$, $p = 1,2,..,s$.
										By applying results proved above for $x \in V^{02}(T)$ with $deg(x) \geq 3$  
										to $y_2$, we obtain that (a) $y_2$ has exactly $2$ neighbors in $V^-(T)$, 
										say, without loss of generality, $z_1, z_2 \in V^-(T)$, and 
										(b) $\gamma_R (T_{z_i} - z_i) = \gamma_R (T_{z_i}) - 1$,  where $i  = 1,2$.
                Recall now that: $h(x) = 2$,  $h|_{T_i}$ is no RD-function on $T_i$  and  $h|_{T_i - y_i}$ is a $\gamma_R$-function on $T_i-y_i$, $i = 1,2$. 
                Hence $h(y_1) = h(y_2) = 0$ and $h|_{T_{z_j}}$ is a  $\gamma_R$-function on $T_{z_j}$, $j \leq s-1$. 
                Since $\gamma_R (T_{z_i} - z_i) = \gamma_R (T_{z_i}) - 1$, $i  = 1,2$, additionally we can choose $h$ so that 
										$h(z_1) = h(z_2) = 1$. But then the function $h_1$ defined as $h_1(u) = h(u)$ when 
										$u \in V(T) - \{y_1,x,y_2,z_1,z_2\}$ and $h_1(y_1) = h_1(x) = 1$, $h_1(y_2) = 2$, $h_1(z_1) = h(z_2) = 0$ 
										is a $\gamma_R$-function on $T$. Now $h_1(x) = 1$,  $h_1(y_2)=2$ and $xy_2 \in E(G)$ lead to a contradiction. 
               	Thus, $N(y_2) = \{x, z\}$.  

										Suppose $z \not \in V^{01}(T)$. Then there is a $\gamma_R$-function $h_4$ on $T$ with $h_4(z) = 2$. 
								     If $h_4(y_2) = 2$, then $h_4(x) = 0$ and the function $h_5$ on $T$ defined as 
										$h_5(x) = h_5(y_2) = 1$ and $h_5(u)  = h_4(u)$ otherwise, is  a $\gamma_R$-function on $T$, a contradiction. 
										Hence $h_4(y_2) = 0$ and since $y_1 \in V^{01}(T)$, $h_4(x) = 2$ and $h_4(y_1) = 0$. 
										But then the function $h_6$ on $T$ defined as 	$h_6(x) = h_6(y_1) = 1$ and $h_6(u)  = h_4(u)$ otherwise,
										is  a $\gamma_R$-function on $T$, a contradiction. 
										Therefore $z  \in V^{01}(T)$, and we are done. 
\medskip
										
			(iii) Assume that $u_1, u_2 \in V^{01}(T)$ are adjacent. 
						Let $T_{u_i}$ be the component of $T-u_1u_2$ that contains 												
       $u_i$, $i=1,2$. Let $g_i$ be a $\gamma_R$-function on $T$ with $g_i(u_i) = 1$, $i =  1,2$.
						Hence $g_i(T_{u_j})$ is a $\gamma_R$-function on $T_{u_j}$, $i,j = 1,2$. 
			Thus $\gamma_R(T) = \gamma_R(T_{u_1}) + \gamma_R(T_{u_2})$. 
			Define now  a $\gamma_R$-function $g_3$ on $T$ as $g_3|_{T_i} = g_i|_{T_i}$, $i= 1,2$. 
			But then a function $g_4$ defined as $g_4(u) = g_3(u)$ when $u \in V(T) - \{u_1, u_2\}$, 
			$g_4(u_1) = 2$ and $g_4(u_2) = 0$ is a $\gamma_R$-function on $T$, 
			contradicting $u_1 \in V^{01}(T)$. Thus $V^{01}(T)$ is independent. 
\end{proof}

\section{Proof of the main result}

\begin{proof}[\bf Proof of Theorem \ref{main}]
Let $T$ be a $\gamma_R$-excellent tree.  
First, we shall prove the following statement.
\begin{itemize}
\item[($\mathcal{P}_2$)] There is a labeling $L:V (T ) \rightarrow\{A, B,C, D\}$ such that
                (a) $L_A(T)$ is either empty or independent,  
                (b) each component of $\left\langle L_B(T) \right\rangle$ and 
                   $\left\langle L_D(T) \right\rangle$ is isomorphic to $K_2$, 
                (c) each element of $L_B(T)$ has degree $2$ and it is adjacent 
                   to exactly one vertex in $L_A(T)$, 
                (d) each vertex $v$ in $L_C(T)$ has exactly $2$ neighbors in $L_A(T) \cup L_D(T)$,  
                      and  if $deg(v) = 2$ then  both  neighbors of $v$ are in $L_D(T)$.
\end{itemize}

We know that  $V(T)  = V^{01}(T) \cup V^{012}(T) \cup V^{02}(T)$.
Define  a labeling $L:V (T ) \rightarrow\{A, B,C, D\}$ by $L_A(T) = V^{01}(T)$, $L_D(T) = V^{012}(T)$, 
$L_B(T) = \{x \in V^{02}(T) \mid deg(x) =2 $ and $|N(x) \cap V^{02}(T)| =1\}$, 
and $L_C(T) = V^{02}(T) - L_B(T)$.  
The validity of $(\mathcal{P}_2)$  immediately follows by Lemma \ref{adj}.

Denote by $\mathscr{T}_1$ the family of all labeled, as in  ($\mathcal{P}_2$), trees $T$. 
 We shall show that if $(T, L)  \in \mathscr{T}_1$ then $(T, L) \in \mathscr{T}$. 
\medskip

 {\bf (I)} {\em Proof of  $(T, L) \in \mathscr{T}_1 \Rightarrow (T, L) \in \mathscr{T}$.}
\medskip

	Let $(T, L) \in \mathscr{T}_1$. The following claim is immediate. 
	\medskip

	{\bf Claim \ref{main}{\bf .1}}		
		\begin{itemize} 
						\item[(i)] Each leaf of $T$ is in $L_A(T) \cup L_D(T)$. 
						\item[(ii)] If $v$ is a stem of $T$, then $v$  is adjacent to at most $2$ leaves. 
						\item[(iii)] If $u_1$ and $u_2$ are leaves adjacent to the same stem, 
						                    then $u_1, u_2 \in L_A(T)$.
						\end{itemize}
		Claim \ref{main}.1  will be used in the sequel without specific reference.				
		We now proceed by induction on $k = |L_B \cup L_C|$. 
		 The base case, $k \leq 2$, is an immediate consequence of the following easy claim, the proof of which is omitted. 
\medskip

	{\bf Claim \ref{main}{\bf .2}}		(see Fig.\ref{fig:Fig 1})
		\begin{itemize} 
						\item[(i)] If $k=0$ then $(T, L) = (H_1, I^1)$. 
						\item[(ii)] If $k=1$ then $(T, L)$ is obtained from $(H_1, I_1)$ by operation $O_2$, i.e. $(T, L) = (H_{11}, I^{11})$. 
						\item[(iii)] If $k=2$ then either $(T, L)$ is    
						                            $(H_r, I^r)$ with $r \in \{2,3,4,5\}$, or  $(T, L)$  is obtained  from $(H_{11}, I^{11})$
						                    by operation $O_1$ or by operation  $O_2$ (see the graphs 
																$(H_s, I^s)$ where  $s \in \{6,7, 8, 9,10\}$. 
						\end{itemize}

									Let $k \geq 3$ and suppose that each tree $(H, L^\prime) \in \mathscr{T}_1$ with 
									$|L^\prime_B(H) \cup L^\prime_C(H)| < k$ is in $\mathscr{T}$.
									Let now $(T,L) \in \mathscr{T}_1$  and $k = |L_B(T) \cup L_C(T)|$. 
									To prove the required result,  it suffices to show that $T$ has a subtree, say $U$, 
									such that $(U, L|_U) \in \mathscr{T}_1$, and $(T, L)$ is obtained from $(U, L|_U)$  
									by one of operations $O_1$,..,$O_{4}$. 
									Consider any diametral path $P: x_1,..,x_n$ in $T$. Clearly $x_1$ is a leaf. 
									Denote by $x_i^1, x_i^2,..$ all neighbors of $x_i$, which do not belong to  $P$, $2 \leq i \leq n-1$.
					\medskip
					
{\bf Case 1}:   $sta(x_1) = A$ and $sta(x_2) = B$. 
                        Then $deg (x_1) = 1$, $deg (x_2) = deg (x_3) =2$, $sta(x_3) = B$ and $sta(x_4) = A$. 
												Thus $T$ is obtained from $T - \{x_1, x_2, x_3\} \in \mathscr{T}_1$ 
												and a copy of $H_2$ by operation $O_3$ (via $x_4$).  	{\tiny\qed} 
		\medskip
										
{\bf Case 2}: $sta(x_1) = A$ and $sta(x_2) = C$.  
                       Hence $deg(x_2) \geq 3$.  By the choice of $P$, $deg(x_2) = 3$, $x_2^1$ is a leaf, 
											$sta(x_2^1) = A$, and $sta(x_3) = C$.  
											If $deg(x_3) \geq 4$ then $T$ is obtained from $T - \{x_2^1, x_1, x_2\} \in \mathscr{T}_1$
											and a copy of $F_1$ by operation $O_1$. So, let $deg(x_3) = 3$. 
											Assume first that $sta(x_4) = A$. 
											Then either $x_3^1$ is a leaf of status $A$ or $x_3^1$ is a stem, $deg(x_3^1) =2$, 
											and both $x_3^1$ and its leaf-neighbor have status $D$. Thus, $T$ is obtained from 
											$T-(N[x_2] \cup N[x_3^1]) \in \mathscr{T}_1$ and 
											a copy  of $H_3$ or $H_4$, respectively, by operation $O_3$ (via $x_4$). 
											Finally let $sta(x_4) = D$.  By the choice of $P$, 
										 either $x_3^1$ is a leaf of status $A$ and then 
											$T$ is obtained from 	$T - (N[x_2] \cup\{ x_3^1\} ) \in \mathscr{T}_1$ 
													and a copy of  $H_3$ by operation $O_4$, 
											or
										 $x_3^1$ is a stem of degree $2$ and both $x_3^1$ and its leaf-neighbor have status $D$, 
										and then  $T$ is obtained from 	$T - \{x_2^1, x_1, x_2\} \in \mathscr{T}_1$
											and a copy of $F_1$ by operation $O_1$. 
														 	{\tiny\qed} 
			\medskip
											
			In what follows, let $sta(x_1) = D$. Hence $deg(x_2) = 2$, $sta(x_2) = D$ and $sta(x_3) = C$. 
			If $deg(x_3) = 2$ then $T$ is obtained from 	$T - N[x_2] \in \mathscr{T}_1$ and  a copy  of $F_4$ by  operation $O_2$.  
							\medskip
											
{\bf Case 3}:  $deg(x_3) = 3$ and $sta(x_4)  \in \{A, D\}$. 
                        In this case $sta(x_3^1) = C$, $x_3^1$ is a stem, $deg(x_3^1) = 3$, 
												and the leaf neighbors of $x_3^1$ 	 have status $A$.  
												Now (a) if $sta(x_4) = A$ then $T$ is obtained from 	$T - (N[x_2] \cup N[x_3^1]) \in \mathscr{T}_1$ 
												and  a copy  of $H_4$ by  operation $O_3$  (via $x_4$), and 
												(b) if $sta(x_4) = D$ then $T$ is obtained from 	$T - (N[x_2] \cup N[x_3^1]) \in \mathscr{T}_1$ 
												and  a copy  of $H_4$ by  operation $O_4$ (via $x_4$). 
	                      {\tiny\qed} 
\medskip

{\bf Case 4}:  $deg(x_3) = 3$, $sta(x_4)  = C$  and $sta(x_3^1) = A$. Hence  $x_3^1$ is a leaf. 
                        If  $deg(x_4) = 3$ and $sta(x_5)  = sta(x_4^1) = D$,  or $deg(x_4) \geq 4$, then 
												 $T$ is obtained from 	$T - \{x_1, x_2, x_3, x_3^1\} \in \mathscr{T}_1$
													and  a copy  of $F_2$ by  operation $O_1$. 
												So, let $deg(x_4) = 3$ and the status of at least one of $x_5$ and $x_4^1$ is $A$. 
												Assume first that $sta(x_4^1) = A$.  Hence $x_4^1$ is a leaf (by the choice of $P$). 
													If $sta(x_5) = A$ then $T$ is obtained from a copy  of $H_4$ and 
														a tree in $ \mathscr{T}_1$  by  operation $O_3$ (via $x_5$).  
														If $sta(x_5) = D$ then $T$ is obtained from a copy  of $H_4$ and 
														a tree in $ \mathscr{T}_1$  by  operation $O_4$ (via $x_5$).
													Second, let 		$sta(x_4^1) = D$. Hence $sta(x_5) = A$,  $deg(x_4^1) = 2$ and 
												the status of the leaf-neighbor of $x_4^1$ is $D$.  
												But then $T$ is obtained from a copy  of $H_5$ and 
														a tree in $ \mathscr{T}_1$  by  operation $O_3$ (via $x_5$).  
  {\tiny\qed} 
\medskip

{\bf Case 5}:  $deg(x_3) = 3$, $sta(x_4) = C$ and $sta(x_3^1) = D$.  
                        Hence $deg(x_3^1) =2$, $x_3^1$ is a stem, 
												and the leaf-neighbor of $x_3^1$ has status $D$. 
												If $deg(x_4) \geq 4$ or $sta(x_5) = sta(x_4^1) = D$, then $T$ is obtained 
												from $T - N[\{x_2, x_3^1\}] \in  \mathscr{T}_1$ and a copy of $F_3$ by operation $O_1$. 
												So, let $deg(x_4) =3$ and  at least one of $x_5$ and $x_4^1$ has  status  $A$.
													Assume $sta(x_4^1) = A$. Hence $x_4^1$ is a leaf. 
												If $sta(x_5) = A$ then $T$ is obtained 
												from $T - N[\{x_2, x_3^1, x_4^1\}] \in  \mathscr{T}_1$ and 
												a copy of $H_6$ by operation $O_3$ (via $x_5$).
													If $sta(x_5) = D$ then $T$ is obtained 
												from $T - N[\{x_2, x_3^1, x_4^1\}] \in  \mathscr{T}_1$ and a copy of $H_6$ by operation $O_4$ (via $x_5$).
												Now let $sta(x_4^1) = D$. Hence $sta(x_5) = A$ and then  $T$ is obtained from a copy  of $H_7$ and 
												a tree in $ \mathscr{T}_1$  by  operation $O_3$ (via $x_5$).  					
  {\tiny\qed} 
\medskip

{\bf Case 6}:  $deg(x_3) \geq 4$. 

                        Hence $x_3$ has a neighbor, say $y$, such that $y \not = x_4$ and  $sta (y) = C$. 
												By the choice of $P$, $y$ is a stem which is adjacent to exactly $2$ leaves, say $z_1$ and $z_2$, 
												and $sta(z_1) = sta (z_2) = A$. But then 		$T$ is obtained 
												from $T - \{y, z_1, z_2\} \in  \mathscr{T}_1$ and a copy of $F_1$ by operation $O_1$.                         
  {\tiny\qed} 
\medskip

{\bf (II)}   {\em Proof of  $(T, S) \in \mathscr{T} \Rightarrow  (T, S) \in \mathscr{T}_1$. }
                   Obvious.  {\tiny\qed} 
\medskip

It remains  the following. 
\medskip

{\bf (III)} {\em Proof of  $(T, S) \in \mathscr{T} \Rightarrow T$  is $\gamma_R$-excellent and ($\mathcal{P}_1$) holds}. 
\medskip

Let $(T, S)  \in \mathscr{T} $. We know that    $(T, S) \in \mathscr{T}_1$.  
We now proceed by induction on $k = |S_B \cup S_C|$. 
First let $k \leq 2$.  By Claim \ref{main}.2, $T \in \mathscr{H} = \{H_1,..,H_{11}\}$. 
It is easy to see that all elements of $\mathscr{H}$ are $\gamma_R$-excellent graphs 
and   ($\mathcal{P}_1$) holds for each $T \in  \mathscr{H}$.

		Let $k \geq 3$ and suppose that if $(H, S^\prime) \in \mathscr{T}$ 
	 and $|S^\prime_B(H) \cup S^\prime_C(H)| < k$, 
	then $H$ is $\gamma_R$-excellent and ($\mathcal{P}_1$) holds with  
	$(T,S)$ replaced by $(H, S^\prime)$. 
			So, let $(T, S) \in \mathscr{T}$  and $k = |S_B(T) \cup S_C(T)|$.				
		 Then there is a $\mathscr{T}$-sequence $\tau: (T^1, S^1), \dots, (T^{j-1}, S^{j-1}), (T, S)$. 
		By induction hypothesis, $T^{j-1}$ is $\gamma_R$-excellent and  ($\mathcal{P}_1$) holds with  
	$(T, S)$ replaced by $(T^{j-1}, S^{j-1})$. 
		We consider four possibilities depending on whether $T$ is obtained from $T^{j-1}$
		by operation  $O_1, O_2, O_3$ or $O_4$.
\medskip

	{\bf Case 7}:  $T$ is obtained from $T^{j-1} \in \mathscr{T}$ and $F_a$ by operation $O_1$, $a \in \{1,2,3\}$. 
		                          Hence $T$ is obtained  after adding the edge $ux$ to the union of $T^{j-1}$ and $F_a$, where 
															$sta_{T^{j-1}}(u) = sta_{F_a}(x) = C$ (see Fig. \ref{fig:F1234}).  
	First note that  $\gamma_R(F_a) = a+1$, and $F_2$ and $F_3$ are $\gamma_R$-excellent graphs. 
		Since $\gamma_R(F_a-x) = 	\gamma_R(F_a)$ and $u \in V^{02}(T^{j-1})$,  Lemma \ref{addedge} 
		implies $\gamma_R(T) = \gamma_R(T^{j-1}) + \gamma_R(F_a)$. 
		Hence for any $\gamma_R$-function $g$ on $T$, the weight of $g|_{F_a}$ 		
		 is not more than $\gamma_R(F_a)$. But then $g(x) \not = 1$ and either  
		$g|_{F_a}$   is a $\gamma_R$-function on $F_a$ or   $g|_{F_a-x}$ is a $\gamma_R$-function on $F_a-x$. 
		By inspection of all $\gamma_R$-functions on $F_a$ and $F_a-x$, we obtain 
		\medskip
		
		\begin{itemize}
		\item[($\alpha_1$)] $S_A(T) \cap V(F_a) = V^{01}(T)\cap V(F_a)$,   $S_B(T) \cap V(F_a) = \emptyset$, 
		                                        $\{x\} = S_C(T) \cap V(F_a) = V^{02}(T)\cap V(F_a)$, and $S_D(T) \cap V(F_a) = V^{012}(T)\cap V(F_a)$. 
		\end{itemize}
		
		By the definition of operation $O_1$ it immediately follows 
		\medskip
		
			\begin{itemize}
		\item[($\alpha_2$)] $S_X(T) \cap V(T^{j-1}) = S_X^{j-1}(T^{j-1})$, for all $X \in \{A, B, C, D\}$. 
		\end{itemize}

		Let $f_1$ be a $\gamma_R$-function on $T^{j-1}$ and $f_2$ a $\gamma_R$-function on $F_a$. 
		Then the RD-function $f$ on $T$ defined as $f|_{T^{j-1}} = f_1$ and $f|_{F_a} = f_2$ is a 
		$\gamma_R$-function on $T$. 
				Since $f_1$ was chosen arbitrarily, we have 
		\begin{itemize}
				\item[($\alpha_3$)]  	$V^{01}(T^{j-1}) \subseteq V^{01}(T) \cup  V^{012}(T)$, 
		$V^{02}(T^{j-1}) \subseteq V^{02}(T) \cup  V^{012}(T)$, and $V^{012}(T^{j-1}) \subseteq  V^{012}(T)$.
		\end{itemize}		
		By ($\alpha_1$) and ($\alpha_3$) we conclude that  $T$ is $\gamma_R$-excellent. 
		\medskip

	Now we shall prove that 
		\begin{itemize}
		\item[($\alpha_4$)] $V^{01}(T)\cap V(T^{j-1}) = V^{01}(T^{j-1})$,   $V^{02}(T)\cap V(T^{j-1}) = V^{02}(T^{j-1})$, 
		                                       and $V^{012}(T)\cap V(T^{j-1}) = V^{012}(T^{j-1})$.		                       
		\end{itemize}

		Assume there is a vertex $z \in V^{02}(T^{j-1}) \cap V^{012}(T)$. 
	By Lemma \ref{adj}, $z$ is adjacent to at most $2$ elements of $V^-(T^{j-1})$. 
	Now by ($\alpha_3$) and since $\Delta(\left\langle V^-(T) \right\rangle) \leq 1$ (by Lemma \ref{no3}), 
	$z$ is adjacent to exactly one element of $V^-(T^{j-1})$. 
	But then Lemma \ref{adj} implies that there is a path $z_1, z, z_2, z_3$ in $T^{j-1}$
	such that 	$deg_{T^{j-1}}(z) =  deg_{T^{j-1}}(z_2) = 2$, $z, z_2 \in V^{02}(T^{j-1})$ and $z_1, z_3 \in V^{01}(T^{j-1})$. 
	Since ($\mathcal{P}_1$) is true for $T^{j-1}$, 	
		$sta_{T^{j-1}}(z_1) = sta_{T^{j-1}}(z_3) = A$,  and 	$sta_{T^{j-1}}(z) =  sta_{T^{j-1}}(z_2) = B$. 
		Clearly, at least one of $z_1$ and $z_3$ is a cut-vertex. 
		Denote by  $Q$ the graph $\left\langle \{z_1, z, z_2, z_3\}\right\rangle$ and 
		let the vertices of $Q$ are labeled as in $T^{j-1}$. 
		Let $U_s$ be the connected component of $T-\{z, z_2\}$, which contains $z_s$, $s=1,3$. 
		
		Assume first that $T^1$ is a subtree of $U \in \{U_1, U_3\}$. 
		Then there is $i$ such that $T^i$ is obtained from $T^{i-1}$ and $Q$ by operation $O_3$.  
		Hence $T^{i-1}$ is a subtree of $U$. Recall that if $y \in V(T^r)$ and $r \leq s \leq j-1$, 
		then $sta_{T^r}(y) = sta_{T^s}(y)$. Using this fact, we can choose $\tau$ so, that $T^{i-1} = U$. 
		Therefore  $U$ is in $\mathscr{T}$. Suppose that neither $z_1$ nor $z_3$ is a leaf of $T^{j-1}$. 
		Define $R^s = T^{i+s} - (V(T^{i-1}) \cup \{z, z_2\})$, $s = 1,..,j-1-i$. 
		Since clearly $R^1$ is in $\{H_2,..,H_7\}$, the sequence $R^1, R^2,..,R^{j-1-i}$ is a $\mathscr{T}$-sequence of 
		$U^\prime$, where $\{U,U^\prime\} = \{U_1, U_2\}$.
		Thus, both $U_1$ and $U_3$ are in $\mathscr{T}$, and $sta_{U_1}(z_1) = A$. 
		By the induction hypothesis, $z_1 \in V^{01}(U_1)$. 
		
		Suppose now that $u \in V(U_3)$. 
		Consider the sequence of trees $U_3, U_4, U_5$, where $U_4$ is obtained from $U_3$ and $Q$ by operation $O_3$ (via $z_3$),  
		and $U_5$  is obtained from $U_4$ and $F_a$ by operation $O_1$. 
		Clearly $U_5$ is in $\mathscr{T}$, $sta_{U_5}(z_1) = A$ and by the induction hypothesis, $z_1 \in V^{01}(U_5)$. 
		Since $T = (U_5\cdot U_1)(z_1)$ and $\{z_1\} = V^{01}(U_1) \cap V^{01}(U_5)$, 
		by Proposition \ref{coales12} we have $z_1 \in V^{01}(T)$. 
		But then Lemma \ref{adj} implies $z_2 \in V^{02}(T)$, a contradiction. 
		
		Now let $u \in V(U_1)$. Denote by $U_2$ the graph obtained from $U_1$ and $F_a$ by operation $O_3$. 
		Then $U_2$ is in $\mathscr{T}$, $sta_{U_2}(z_1) = A$, and by induction hypothesis, $z_1 \in V^{01}(U_2)$. 
		Define also the graph $U_6$ as obtained from $U_3$ and $Q$ by operation $O_3$, i.e. 
		$U_6 = (U_3 \cdot Q)(z_3)$. Then $U_6$ is in $\mathscr{T}$, $sta_{U_6}(z_1) = A$ and 
		by induction hypothesis, $z_1 \in V^{01}(U_6)$. 
		Now by Proposition \ref{coales12}, $z_1 \in V^{01}(T)$, 
		which leads to $z_2 \in V^{02}(T)$ (by Lemma \ref{adj}), a contradiction. 
		
		Thus, in all cases we have a contradiction. 
		Therefore $V^{02}(T^{j-1}) \subseteq V^{02}(T)$ when both $z_1$ and $z_3$ are cut-vertices. 
		If $z_1$ or $z_3$ is a leaf, then, by similar arguments,
		we can obtain the same result. 
		
		Let now $T^1 \equiv Q$. Then $T^2$ is obtained from $T^1$ and $H_k$ by operation $O_3$. 
		Consider the sequence of trees $\tau_1: T^1_1 = H_k, T^2, ..., T^{j-1}$. 
		Clearly $\tau_1$ is a $\mathscr{T}$-sequence of $T^{j-1}$ and $T^1_1 \not = Q$. 
				Therefore we are in the previous case. 
				Thus, $V^{02}(T^{j-1}) = V(T^{j-1}) \cap V^{02}(T)$. 
				
				Assume now that there is a vertex $w \in V^{01}(T^{j-1}) \cap V^{012}(T)$. 
			By Lemma \ref{adj}(i)   $w$ has  a neighbor in $T$, say $w^\prime$, 
			such that $w^\prime \in V^{012}(T)$. Since $w \not \equiv u$, $w^\prime \in V(T^{j-1})$.  
			But all neighbors of $w$ in $T^{j-1}$ 	are in $V^{02}(T^{j-1})$
			(by Lemma \ref{adj} applied to $T^{j-1}$ and $w$). 
			Since $V^{02}(T^{j-1}) = V(T^{j-1}) \cap V^{02}(T)$, we obtain a contradiction. 
\medskip		
				
		Thus ($\alpha_4$) is true. 
	
 Now we are prepared to prove that $(\mathcal{P}_1)$ is valid. 
Using, in the chain of equalities below, consecutively  ($\alpha_2$), the induction hypothesis, 
($\alpha_1$) and ($\alpha_4$),  we obtain 
\[
     S_A(T) = S_A^{j-1}(T^{j-1}) \cup (S_A(T) \cap V(F_a)) = V^{01}(T^{j-1}) \cup (V^{01}(T) \cap V(F_a)) = V^{01}(T),
\] 
and similarly, $S_D(T) = V^{012}(T)$.  
Since $u \not\in S_B(T)$ and $S_B(T) \cap V(F_a) = \emptyset$, we have 
\begin{align*}
S_B(T)&= S_B(T) \cap V(T^{j-1}) \stackrel{(\alpha_2)}{=} S_B^{j-1}(T^{j-1}) \\
&=  \{t \in V^{02}(T^{j-1}) \mid deg_{T^{j-1}}(t) =2  {\rm\ and\ }  |N_{T^{j-1}}(t) \cap V^{02}(T^{j-1})|=1\}  \notag \\
&\stackrel{(\alpha_4)}{=} 	\{t \in V^{02}(T) \cap V(T^{j-1}) \mid deg_{T}(t) =2  {\rm\ and\ }  |N_{T}(t) \cap V^{02}(T)| =1\}  \notag \\
&= \{t \in V^{02}(T)  \mid deg_{T}(t) =2  {\rm\ and\ }  |N_{T}(t) \cap V^{02}(T)| =1\}.
\end{align*}

The last equality follows from $deg_T(x) > 2$ and $\{x\} = V^{02}(T) \cap V(F_a)$ (see ($\alpha_1$)). 
Now the equality $S_C(T) = V^{02}(T) - S_B(T)$ is obvious. 
Thus, $(\mathcal{P}_1)$ holds and we are done. 
\medskip

{\bf Case 8}:  $T$ is obtained from $T^{j-1} \in \mathscr{T}$  by operation $O_2$. 
 Clearly $\gamma_R(F_4) = \gamma_R(F_4-x) = 2$. 
By Lemma \ref{addedge}, $\gamma_R(T) = \gamma_R(T^{j-1}) + \gamma_R(H_4)$. 
Let $f_1$ be a $\gamma_R$-function on $T^{j-1}$ and $f_2$ a $\gamma_R$-function on $F_4$. 
Then the function $f$ defined as $f|_{T^{j-1}} = f_1$ and $f|_{F_4} = f_2$ is a $\gamma_R$-function on $T$. 
Therefore $V^{012}(T^{j-1}) \subseteq V^{012}(T)$, 
$V^{01}(T^{j-1}) \subseteq V^{01}(T) \cup V^{012}(T)$, and $V^{02}(T^{j-1}) \subseteq V^{02}(T) \cup V^{012}(T)$.

Assume that there is $y \in V^{0s}(T^{j-1}) \cap V^{012}(T)$, $s \in \{1,2\}$, and let $f^\prime$ be  
 a $\gamma_R$-function on $T$ with  $f^\prime(y) = r \not\in \{0,s\}$. 
If $f^\prime|_{T^{j-1}}$ is an RD-function on $T^{j-1}$, then $f^\prime|_{T^{j-1}}(V(T^{j-1})) > \gamma_R(T^{j-1})$ 
and $f^\prime|_{F_4}(V(F_4)) \geq 2$. This leads to $f^\prime(V(T)) > \gamma_R(T)$, a contradiction. 
Hence $f^\prime|_{T^{j-1}}$ is no RD-function on $T^{j-1}$ and $f^\prime|_{T^{j-1} - u}$ 
is a $\gamma_R$-function on $T^{j-1} - u$. 
Define now an RD-function $f^{\prime\prime}$ on $T^{j-1}$ as $f^{\prime\prime}|_{T^{j-1} - u} =  f^\prime|_{T^{j-1}-u}$
 and $f^{\prime\prime}(u) = 1$.  Since $u \in V^-(T^{j-1})$,  $f^{\prime\prime}$ 
is a $\gamma_R$-function on $T^{j-1}$ with $f^{\prime\prime}(y) = r \not \in \{0,s\}$, 
 a contradiction with $y \in V^{0s}(T^{j-1})$. Thus 
\begin{itemize}
\item[($\alpha_5$)] $V^{012}(T^{j-1}) = V^{012}(T) \cap V(T^{j-1}) $, $V^{01}(T^{j-1}) = V^{01}(T) \cap V(T^{j-1}) $, 
                       and $V^{02}(T^{j-1}) = V^{02}(T) \cap V(T^{j-1}) $.  
\end{itemize}

Let $x, x_1, x_2$ be a path in $F_4$, $h_1$ a  $\gamma_R$-function on $T^{j-1}$ with $h_1(u) = 2$, 
and $h_2$ a $\gamma_R$-function on $T^{j-1} - u$. 
				Define $\gamma_R$-functions $g_1,..,g_4$  on $T$ as follows: 
				\begin{itemize}
				\item[$\bullet$] $g_1|_{T^{j-1}} = h_1$, $g_1(x) = g_1(x_2) = 0$ and $g_1(x_1) =2$;	
			 \item[$\bullet$] $g_2|_{T^{j-1}} = h_1$, $g_2(x) = 0$ and $g_2(x_1) = g_2(x_2) = 1$;	
			 \item[$\bullet$] $g_3|_{T^{j-1}} = h_1$, $g_3(x) = g_3(x_1) = 0$ and $g_3(x_2) = 2$;
				 \item[$\bullet$] $g_4|_{T^{j-1}-u} = h_2$, $g_4(u) = g_4(x_1) = 0$, $g(x) = 2$ and $ g_4(x_2) = 1$.
			\end{itemize}
			This,  ($\alpha_5$) and Lemma \ref{no3} allows us to conclude that  $T$ is $\gamma_R$-excellent, 
			$x_1,x_2 \in V^{012}(T)$ and $x \in V^{02}(T)$. 
	
	By induction hypothesis,   ($\mathcal{P}_1$) holds with  $(T, S)$ replaced by $(T^{j-1}, S^{j-1})$. Then 
Since $u \not\in S_B(T)$ and $S_B(T) \cap V(F_4) = \emptyset$, we have 
\begin{align*}
 S_B(T) & =  S_B^{j-1}(T^{j-1}) \\
&= \{t \in V^{02}(T^{j-1}) \mid deg_{T^{j-1}}(t) =2  {\rm\ and\ }  |N_{T^{j-1}}(t) \cap V^{02}(T^{j-1})|=1\} \\
& = \{t \in V^{02}(T)  \mid deg_{T}(t) =2  {\rm\ and\ }  |N_{T}(t) \cap V^{02}(T)| =1\}.
\end{align*}
The last equality follows from $deg_T(x) > 2$ and $\{x\} = V^{02}(T) \cap V(F_4)$. 
Now  the equality $S_C(T) = V^{02}(T) - S_B(T)$ is obvious. 
Thus, $(\mathcal{P}_1)$ is true. 
\medskip

{\bf Case 9}:  $T$ is obtained from $T^{j-1} \in \mathscr{T}$  by operation $O_3$.
                        Say $T = (T^{j-1}\cdot H_k)(u,v:u)$, where $sta_{T^{j-1}}(u) = sta_{H_k}(v) = sta_T(u) = A$ and $k \in \{2,..,7\}$. 
												Hence $S_X(T) = S_X^{j-1}(T^{j-1}) \cup I_X^k(H_k)$, for any $X \in \{A, B, C, D\}$. 
												We know that ($\mathcal{P}_1$) holds with $(T, S)$ replaced by any of $(T^{j-1}, S^{j-1})$ and $(H_k, I^k)$. 
												Hence $S_A(T)  = S_A^{j-1}(T^{j-1}) \cup I_A^k(H_k) = V^{01}(T^{j-1}) \cup V^{01}(H_k)$. 
												Now, by Proposition \ref{coales12}, applied to $T^{j-1}$ and $H_k$, $S_A(T) = V^{01}(T)$.
												Similarly we obtain $S_D(T) = V^{012}(T)$. 
                         We also have 
				\begin{align*}										
						S_B(T)&  = S_B^{j-1}(T^{j-1}) \cup I_B^k(H_k)\\
							          & = 	\{t \in V^{02}(T^{j-1}) \mid deg_{T^{j-1}}(t) =2   {\rm\ and\ } |N_{T^{j-1}}(t) \cap V^{02}(T^{j-1})| =1\} \\ 
												& \cup  \{t \in V^{02}(H_k) \mid deg_{{H_k}}(t) =2   {\rm\ and\ } |N_{H_k}(t) \cap V^{02}(H_k)| =1\}\\
												& = \{t \in V^{02}(T^{j-1}) \cup V^{02}(H_k) \mid deg_T(t) =2   {\rm\ and\ }  |N_T(t) \cap V^{02}(T)| =1\},
				\end{align*}									
																										as required, because $V^{02}(T^{j-1}) \cup V^{02}(H_k)  = V^{02}(T)$ (by Proposition \ref{coales12}).  
         Now the equality $S_C(T) = V^{02}(T) - S_B(T)$ is obvious. 
\medskip

{\bf Case 10}:  $T$ is obtained from $T^{j-1} \in \mathscr{T}$  and $H_k \in \mathscr{T}$, $k \in \{3,4,6\}$, by operation $O_4$.
By Lemma \ref{01012} and induction hypothesis, $\gamma_R (T) = \gamma_R (T^{j-1}) + \gamma_R (H_k)  - 1$ and $u \in V^{012}(T)$. 
Let $f_1$ be a $\gamma_R$-function on $T^{j-1}$ and $f_2$ a $\gamma_R$-function on $H_k - v$. 
Then the function $f$ defined as $f|_{T^{j-1}} = f_1$ and $f|_{H_k-v} = f_2$ is a $\gamma_R$-function on $T$. 
Therefore $V^{012}(T^{j-1}) \subseteq V^{012}(T)$, 
$V^{01}(T^{j-1}) \subseteq V^{01}(T) \cup V^{012}(T)$, and $V^{02}(T^{j-1}) \subseteq V^{02}(T) \cup V^{012}(T)$.
Assume that there is $y \in V^{0s}(T^{j-1}) \cap V^{012}(T)$, $s \in \{1,2\}$, and let $f^\prime$ be  
 a $\gamma_R$-function on $T$ with  $f^\prime(y) = r \not\in \{0,s\}$. 
But then $f^\prime|_{T^{j-1}}$ is no RD-function on $T^{j-1}$, $f^\prime(u) = 0$, 
$f^\prime|_{T^{j-1}-u}$ is a $\gamma_R$-function on $T^{j-1}-u$ and 
$f^\prime|_{H_k}$ is a $\gamma_R$-function on $H_k$.
Define now an RD-function $g_1$ on $T^{j-1}$ as $g_1|_{T^{j-1} - u} =  f^\prime|_{T^{j-1}-u}$ and $g_1(u) = 1$. 
Since $g_1(V(T^{j-1}) )= \gamma_R(T^{j-1}-u) + 1 = \gamma_R(T^{j-1})$, $g_1$ is a  
$\gamma_R$-function on  $T^{j-1}$. But $g_1(y) = r \not \in \{0,s\}$,  a contradiction.
Thus 
\begin{itemize}
\item[($\alpha_6$)] $V^{012}(T^{j-1}) = V^{012}(T) \cap V(T^{j-1}) $, $V^{01}(T^{j-1}) = V^{01}(T) \cap V(T^{j-1}) $, 
                       and $V^{02}(T^{j-1}) = V^{02}(T) \cap V(T^{j-1}) $.  
\end{itemize}
The next claim is obvious.
\medskip

	{\bf Claim \ref{main}{\bf .3}}	
	Let $x$ be the neighbor of $v$ in $H_k$, $k \in \{3,4,6\}$.
	Then $\gamma_R(H_3) = 4$, $\gamma_R(H_4) = 5$, $\gamma_R(H_6) = 6$, 
	 $\gamma_R(H_k-v) = 	\gamma_R(H_k - \{v,x\}) = \gamma_R(H_k)$, 
	and $l(x) = 0$ for any $\gamma_R$-function on $H_k-v$. 
	\medskip
	
		Let $h$ be a $\gamma_R$-function on $T$. 
		We know that $u \in V^{012}(T)$, $u \in V^{012}(T^{j-1})$, $v \in V^{01}(H_k)$, 
		and $\gamma_R (T) = \gamma_R (T^{j-1}) + \gamma_R (H_k)  - 1$. 		
		Then by Claim \ref{main}.3 we clearly have:
\begin{itemize}
		\item[(a1)] If $h(u) = 2$ then at least one of the following holds: 
		    \begin{itemize}
		         \item[(a1.1)] $h|_{H_k-v}$ is a $\gamma_R$-function on $H_k-v$, and 
						   \item[(a1.2)]  $h|_{H_k-\{v,x\}}$ is a $\gamma_R$-function on $H_k-\{v,x\}$. 
	      \end{itemize}
			\item[(a2)] If $h(u) = 1$ then $h|_{H_k-v}$ is a $\gamma_R$-function on $H_k-v$.
				\item[(a3)] 	If $h(u) = 0$ then	either $h|_{H_k}$ is a $\gamma_R$-function on $H_k$, or 
				                      $h|_{H_k-v}$ is a $\gamma_R$-function on $H_k-v$. 
	\end{itemize}

	Let $l_1$, $l_2$, $l_3$, $l_4$ and $l_5$ be $\gamma_R$-functions on
	$H_k, H_k-v$, $H_k-\{v,x\}$, $T^{j-1} - u$ and $T^{j-1}$, respectively, 
	and let $l_5(u) = 2$.  
	Define the functions $h_1$, $h_2$, and  $h_3$ on $T$ as follows: 
	(i) $h_1|_{T^{j-1}} = l_5$, $h_1(x) = 0$ and $h_1|_{H_k-\{v,x\}} = l_3$, 
	(ii) $h_2|_{T^{j-1}} = l_5$ and $h_1|_{H_k-v} = l_2$, and 
	(iii) $h_3|_{T^{j-1}-u} = l_4$ and $h_3|_{H_k} = l_1$. 
	Clearly $h_1$, $h_2$, and  $h_3$ are $\gamma_R$-functions on $T$.
		After inspection of all $\gamma_R$-functions of  $H_k, H_k-v$  and  $H_k-\{v,x\}$, 
		we conclude that  
		$V^{01}(H_k) - \{v\} \subseteq V^{01}(T)$, $V^{02}(H_k) \subseteq V^{02}(T)$, 
		and $V^{012}(H_k) \subseteq V^{012}(T)$. This and ($\alpha_6$)  imply 
  \begin{itemize}
\item[($\alpha_7$)] $V^{012}(T) = V^{012}(T^{j-1})  \cup V^{012}(H_k)$, $V^{02}(T) = V^{02}(T^{j-1})  \cup V^{02}(H_k)$, 
                        and $V^{01}(T) = V^{01}(T^{j-1})  \cup (V^{01}(H_k) - \{v\})$.
\end{itemize}		
Since  ($\mathcal{P}_1$) holds with  	$T$ replaced by  $H_k$  or by $T^{j-1}$ (by induction hypothesis), 
using  ($\alpha_7$) we obtain that  ($\mathcal{P}_1$) is satisfied. 
\end{proof}

\section{Corollaries}

The next three results immediately follow by Theorem \ref{main}. 

\begin{corollary}\label{unilab}
If $(T,S_1),(T,S_2)\in \mathscr{T}$ then $S_1 \equiv S_2$.
\end{corollary}

If $(T, S) \in \mathscr{T}$ then we call $S$ the $\mathscr{T}$-{\em labeling} of $T$. 

\begin{corollary} \label{boundextrees}
Let $T$ be a $\gamma_R$-excellent tree of order $n \geq 5$, and $S$ the  $\mathscr{T}$-{\em labeling} of $T$. 
Then $\frac{n}{5} \leq |V^{02}(T)| \leq \frac{2}{3} (n-1)$ and $\frac{4}{5}n \geq |V^-(T)| \geq \frac{1}{3} (n+2)$.
Moreover, 
\begin{itemize}
\item[(i)] $\frac{n}{5} = |V^{02}(T)|$ if and only if $(T, S)$ has a $\mathscr{T}$-sequence  
                  $\tau: (T^1, S^1), \dots, (T^j, S^j)$, such that  $(T^1, S^1) = (F_3, J^3)$ 
                   and if $j \geq 2$,  $(T^{i+1}, S^{i+1})$ can be obtained recursively from 
                  $(T^i, S^i)$ and $(F_3, J^3)$  by operation $O_1$. 
\item[(ii)]   $|V^{02}(T)| \leq \frac{2}{3} (n-1)$ if and only if $(T, S)$ has a $\mathscr{T}$-sequence  
                  $\tau: (T^1, S^1),..,$ \\ $(T^j, S^j)$, such that  $(T^1, S^1) = (H_2, I^2)$ 
                   and if $j \geq 2$,  $(T^{i+1}, S^{i+1})$ can be obtained recursively from 
                  $(T^i, S^i)$ and $(H_2, I^2)$  by operation $O_3$. 
\end{itemize}
\end{corollary}

\begin{corollary} \label{minedge}
Let $G$ be an $n$-order  $\gamma_R$-excellent  connected graph of minimum size. 
Then either $G=K_3$ or $n \not= 3$ and $G$ is a tree.
\end{corollary}


\section{Special cases}\label{aaa}
Let $G$  be a graph and  $\{a_1,..,a_k\} \subseteq \{0, 1,2,01,02,12,012\}$.  
We say that  $G$ is  a $\mathcal{R}_{a_1,..,a_k}$-graph if $V(G) = \cup_{i=1}^kV^{a_i}(G)$ and 
all $V^{a_1}(G),..,V^{a_k}(G)$ are nonempty. Now let $T$ be a $\gamma_R$-excellent tree of order at least $2$.  
By Theorem \ref{main}, we immediately conclude that 
$T \in \mathcal{R}_{012} \cup \mathcal{R}_{01, 02} \cup \mathcal{R}_{02, 012} \cup \mathcal{R}_{01, 02, 012}$. 
Moreover, 
\begin{itemize}
\item[(i)] $T \in \mathcal{R}_{012}$ if and only if $T = K_2$, and 
\item[(ii)] $T \in \mathcal{R}_{01, 02, 012}$ if and only if none of $S_A(T),  S_C(T)$ and 
                    $S_D(T)$ is empty, where $S$ is the $\mathscr{T}$-labeling of $T$. 
\end{itemize}
In this section, we turn our attention to the classes $\mathcal{R}_{01, 02}$ and $\mathcal{R}_{02, 012}$. 
\medskip

\subsection{$\mathbf{\mathcal{R}_{01, 02}}$-graphs.}

Here  we give necessary and sufficient conditions for a tree to be in $\mathcal{R}_{01,02}$. 
We  define a subfamily $\mathscr{T}_{01,02}$ of $\mathscr{T}$ as follows. 
A labeled tree  $(T, S) \in \mathscr{T}_{01,02} $
if and only if $(T, S)$  can be obtained from a sequence of labeled trees 
$\tau: (T^1, S^1), \dots, (T^j, S^j)$, ($j \geq 1$),  such that  $(T^1, S^1)$ is in 
$\{(H_2, I^2), (H_3, I^3)\}$ (see  Figure \ref{fig:Fig 1}) and $(T, S) = (T^j, S^j)$, 
and, if $j \geq 2$,  $(T^{i+1}, S^{i+1})$ can be obtained recursively from $(T^i, S^i)$ 
 by one of the  operations $O_5$ and $O_6$ listed below; 
 in this case  $\tau$ is said to be  a  $\mathscr{T}_{01,02}$-{\em sequence} of $T$. 
\medskip
 \\
{\bf Operation}  {$O_5$}.  The labeled tree $(T^{i+1},S^{i+1})$ is obtained from 
$(T^i,S^i)$ and $(F_1, J^1)$ (see  Figure \ref{fig:F1234})  by adding
 the edge $ux$, where $u \in V (T_i)$, $x \in V(F_1)$ and  $sta_{T^i}(u) = sta_{F_1}(x) = C$.
\medskip
 \\
{\bf Operation}   {$O_6$}.   The labeled tree $(T^{i+1},S^{i+1})$ is obtained from 
 $(T^i,S^i)$ and $(H_k, I^k)$, $k \in \{2,3\}$ (see  Figure \ref{fig:Fig 1}),   
in such a way that  $T^{i+1}= (T^i \cdot H_k)(u,v:u)$, 
where $sta_{T^i}(u) = sta_{H_k}(v) = A$, and $sta_{T^{i+1}}(u) = A$.
\medskip

Remark that once a vertex is assigned a status, 
this status remains unchanged as the labeled tree $(T, S)$ is recursively constructed. 
By the above definitions we see that $S_D(T)$ is empty when $(T,S) \in \mathscr{T}_{01,02}$.
So, in this case, it is naturally to consider a labeling $S$ as $S:V (T ) \rightarrow\{A,B,C\}$.   
 From Theorem \ref{main} we immediately obtain the following result. 
\begin{corollary}\label{0102c}
Let $T$ be a  tree of order at least $2$.
Then $T \in \mathcal{R}_{01,02}$ if and only if there is a labeling
$S:V (T ) \rightarrow\{A,B,C\}$ such that  $(T, S)$ is in $\mathscr{T}_{01,02} $.
Moreover, if $(T, S) \in \mathscr{T}_{01,02}$ then 
	\begin{itemize}
	\item[($\mathcal{P}_3$)] 	 $S_B(T) = \{x \in V^{02}(T) \mid deg(x) =2 $ and $|N(x) \cap V^{02}(T)| =1\}$, 
	$S_A(T) = V^{01}(T)$, and $S_C(T) = V^{02}(T) - S_B(T)$.  		
\end{itemize}
\end{corollary}

As un immediate consequence of Corollary \ref{unilab} we obtain: 
\begin{corollary}\label{unilab1}
If $(T,S_1),(T,S_2)\in \mathscr{T}_{01,02}$ then $S_1 \equiv S_2$.
\end{corollary}

A  graph $G$ is called a $\overline{K_2}$-{\em corona} if  
each  vertex of $G$ is either a stem or a leaf, and  each stem of $G$ is adjacent to exactly $2$ leaves. 
In a {\em labeled} $\overline{K_2}$-{\em corona} all  leaves have status $A$ and all  stems have status $C$. 
   
\begin{proposition} \label{forb}
Every connected $n$-order graph $H$, $n \geq 2$, is an induced subgraph of a $\mathcal{R}_{01,02}$-graph 
with the domination number equals to $2|V(H)|$. 
\end{proposition}
\begin{proof}
Let  a graph $G$ be  a $\overline{K_2}$-corona such that 
 the induced subgraph by the set of all stems of $G$ is isomorphic to $H$. 
  Let $x$ be a stem of $G$ and $y,z$ the leaf neighbors of $x$ in $G$. 
	Then clearly for any $\gamma_R$-function $f$ on $G$, $f(x) + f(y) + f(z) \geq 2$, 
	$f(y) \not= 2 \not= f(z)$ and $f(x) \not= 1$.  
	Define RD-functions $h$ and $g$ on $G$ as follows: 
	(a) $h(u) = 2$ when $u$ is a stem of $G$ and $h(u) = 0$, otherwise, and 
	(b) $g(v) = h(v)$ when $v \not\in \{x, y, z\}$, and $g(x) = 0$, $g(y) = g(z) = 1$. 
	Therefore $\gamma_R(G) = 2|V(H)|$ and $G$ is in $\mathcal{R}_{01,02}$. 
\end{proof}
	
\begin{corollary}\label{forb1}
There does not exist a forbidden subgraph characterization of the class of $\mathcal{R}_{01,02}$-graphs. 
There does not exist a forbidden subgraph characterization of the class of $\gamma_R$-excellent graphs.
\end{corollary}	

Let $\mathscr{T}_{01,02} ^\prime$ be the family of all labeled trees $(T, L)$
 that can be obtained from a sequence of labeled trees 
$\lambda: (T^1, L^1), \dots, (T^j, L^j)$, ($j \geq 1$),  such that  $(T, L) = (T^j, L^j)$, 
 $(T^1, L^1)$ is either $(H_2, I^2)$ (see  Figure \ref{fig:Fig 1})  or a labeled  $\overline{K_2}$-corona tree, 
and, if $j \geq 2$,  $(T^{i+1}, L^ {i+1})$ can be obtained recursively from $(T^i, L^i)$ 
 by one of the  operations $O_7$ and $O_8$ listed below;
in this case  $\lambda$ is said to be  a  $\mathscr{T}_{01,02} ^\prime$-{\em sequence}   of $T$. 
\medskip
 \\
{\bf Operation}   {$O_7$}.   The labeled tree $(T^{i+1},L^{i+1})$ is obtained from 
 $(T^i,L^i)$ and $(H_2, I^2)$,   in such a way that  $T^{i+1}= (T^i \cdot H_2)(u,v:u)$, 
where $sta_{T^i}(u) = sta_{H_2}(v) = A$, and $sta_{T^{i+1}}(u) = A$.
\medskip
\\
{\bf Operation}   {$O_8$}.   The labeled tree $(T^{i+1},L^{i+1})$ is obtained from 
 $(T^i,L^i)$ and  a labeled  $\overline{K_2}$-corona tree, say $U_i$, 
 in such a way that  $T^{i+1}= (T^i \cdot U_i)(u,v:u)$, 
where $sta_{T^i}(u) = sta_{U_i}(v) = A$, and $sta_{T^{i+1}}(u) = A$.
\medskip
 
Again, once a vertex  is assigned a status, 
this status remains unchanged as the $2$-labeled tree $T$ is recursively constructed.

\begin{theorem}\label{main1}
For any tree $T$ the following are equivalent.
\begin{itemize}
\item[{\ ($A_1$)}]   $T$ is in $\mathcal{R}_{01,02}$. 
\item[{\ ($A_2$)}]  There is a labeling $S:V (T ) \rightarrow\{A,B,C\}$ such that  $(T, S)$ is in $\mathscr{T}_{01,02} $.
\item[{\ ($A_3$)}]  There is a labeling $L:V (T ) \rightarrow\{A,B,C\}$ such that  $(T,L)$ is in $\mathscr{T}_{01,02}^\prime $.
\end{itemize}
\end{theorem}
\begin{proof}[\bf Proof.]
($A_1$) $\Leftrightarrow$ ($A_2$): By Corollary \ref{0102c}. 
\medskip

($A_3$) $\Rightarrow$ ($A_2$): 

  Let a tree  $(T, L) \in \mathscr{T}_{01,02} ^\prime$. 
	It is clear that all $\mathscr{T}_{01,02} ^\prime$-sequences of $(T, L)$ have the same number of elements. 
	Denote this number by $r(T)$. 		
	We shall prove that $(T, L) \in \mathscr{T}_{01,02} ^\prime   \Rightarrow (T, L) \in \mathscr{T}_{01,02}$.
	We proceed by induction on $r(T)$. 	
		 If $r(T)=1$ then    either $(T, L)$ is  a labeled  $\overline{K_2}$-corona tree, 
		or $(T, L) = (H_2, I^2)$. In both cases $(T, L)  \in \mathscr{T}_{01,02}$. We need the following obvius claim. 
	\medskip
			
{\bf Claim \ref{main1}.1}			
			If $(T^\prime, L^\prime)$ is a labeled  $\overline{K_2}$-corona tree, 
			$w \in V(T^\prime)$ and $sta (w) = A$, then 		
		 either $(T^\prime, L^\prime)$ is $(H_3, I^3)$ or there is 
 a  $\mathscr{T}$-sequence $\tau: (T^1, S^1), \dots, (T^l, S^l)$, ($l \geq 2$),  
		such that  $(T^1, S^1) = (H_3, I^3)$,  $w \in V(T^1)$, $(T^l, S^l) = (T^\prime, L^\prime)$, and   
		$(T^{i+1}, S^{i+1})$ can be obtained recursively from $(T^i, S^i)$ and $(F_1, J^1)$ 
    by   operation $O_5$. 
				\medskip
									
				Suppose now that  each tree $(H, L_H) \in \mathscr{T}_{01,02}^\prime$ with 
$r(H)  <k$ is in $\mathscr{T}_{01,02}$, where $k \geq 2$. 
Let $\lambda: (T^1, L^1), \dots, (T^k, L^k)$, be  a  $\mathscr{T}_{01,02} ^\prime$-sequence
of a labeled tree $(T, L) \in \mathscr{T}_{01,02}^\prime$. 
By the induction hypothesis, $(T^{k-1}, L^{k-1})$ is in $\mathscr{T}_{01,02}$. 
Let $\tau_1: (U^1, S^1), \dots, (U^m, S^m)$ be a $\mathscr{T}$-sequence of $(T^{k-1}, L^{k-1})$. 
Hence $U^m = T^{k-1}$ and $S^m = L^{k-1}$. 
If $ (T^k, L^k)$ is obtained from $ (T^{k-1}, L^{k-1})$ and $(H_2, I^2)$ by operation $O_7$, then 
$(U^1, S^1), \dots, (U^m, S^m), (T^k, L^k) = (T, L)$ is a $\mathscr{T}$-sequence of $(T, L)$. 
So, let $ (T^k, L^k)$ is obtained from $ (T^{k-1}, L^{k-1})$ and  a labeled  $\overline{K_2}$-corona tree, 
say $(Q, L_q)$  by operation $O_8$. Hence $T^{k-1}$ and $Q$ have exactly one vertex in comman, say $w$,  
and $sta_{T^{k-1}}(w) = sta_Q(w) = sta_{T^k}(w)= A$. 
By Claim \ref{main1}.1, $(Q, L_q) \in \mathscr{T}_{01,02}$ 
and it has a $ \mathscr{T}_{01,02}$-sequence, say  
$(Q^1, L_q^1),\dots, (Q^s, L_q^s)$ such that $Q^s=Q$, $L_q = L_q^s$, and $w \in V(Q^1)$. 
Denote $W^{m+i} = \left\langle V(U^m) \cup V(Q^i ) \right\rangle$ and let a labeling 
$S^{m+i}$ be such that $S^{m+i}|_{U^m} = S^m$ and $S^{m+i}|_{Q^i} = L_q^i$. 
Then the sequence of labeled trees
 $(U^1, S^1), \dots , (U^m, S^m)$, $(W^{m+1}, S^{m+1}), \dots, (W^{m+s},S^{m+s}) = (T, L)$ 
is a $\mathscr{T}_{01,02}$-sequence of $(T,L)$.
\medskip

($A_2$) $\Rightarrow$ ($A_3$): 

Let a labeled tree $(T, S)  \in \mathcal{T}_{01, 02}$. Then $(T, S) $ has 
a $\mathscr{T}$-sequence $\tau : (T^1, S^1), \dots, (T^j, S^j) = (T, S)$, 
where $(T^1, S^1) \in \{(H_2, I^2), (H_3, I^3)\} \subset \mathscr{T}_{01, 02}^\prime$.
We proceed by induction on $p(T) = \Sigma_{z \in \mathtt{C}(T)} deg_T(z)$, 
where $\mathtt{C}(T)$ is the set of all cut-vertices of $T$ that belong to $S_A(T)$. 
 Assume first $p(T) = 0$.  If $j=1$ then we are done. If $j \geq 2$ then 
 $(T^1, S^1) = (H_3, I^3)$ and $(T^{i+1}, S^{i+1})$ is obtained from $(F_1, J^1)$ and $(T^i , S^i)$ by operation $O_5$. 
Thus, $(T, S)$ is a labeled $\overline{K_2}$-corona tree, which allow us to conclude that 
$(T, S)$ is in $\mathscr{T}_{01,02}^\prime$.

Suppose now that $p(T) = k \geq 1$ and for each labeled tree $(H, S_H) \in \mathscr{T}_{01,02}$ with $p(H)  < k$
is fulfilled $(H, S_H) \in \mathscr{T}_{01,02}^\prime$. 
 Then there is a cut-vertex, say $z$, such that
 (a) $z \in S_A(T)$,  
(b) $(T, S)$ is a coalescence of $2$ graphs, say $(T^\prime, S|_{T^\prime})$
      and $(T^{\prime\prime}, S|_{T^{\prime\prime}})$, via $z$, 
and (c) no vertex in $S_A(T) \cap V(T^{\prime\prime})$ is a cut-vertex of $T^{\prime\prime}$. 
 Hence $(T^\prime, S|_{T^\prime}) \in \mathscr{T}_{01,02}^\prime$ (by induction hypothesis) 
and $(T^{\prime\prime}, S|_{T^{\prime\prime}})$ 
is either a labeled $\overline{K_2}$-corona tree or $H_2$. 
 Thus $(T,S)$ is in $\mathscr{T}_{01,02}^\prime $.
\end{proof}

\subsection{$\mathbf{\mathcal{R}_{02, 012}}$-trees.}

Our aim in this section is to present a characterization of $\mathcal{R}_{02,012}$-trees. 
For this purpose, we need the following definitions.
Let $\mathscr{T}_{02,012} \subset \mathscr{T}$ be such that $(T, S) \in \mathscr{T}_{02,012}$
if and only if $(T, S)$  can be obtained from a sequence of labeled trees 
$\tau: (T^1, S^1), \dots, (T^j, S^j)$, ($j \geq 1$),  such that  $(T^1, S^1) = (F_3, J^3)\}$ (see  Figure \ref{fig:F1234})
 and $(T, S) = (T^j, S^j)$, and, if $j \geq 2$,  $(T^{i+1}, S^{i+1})$ can be obtained recursively 
from $(T^i, S^i)$  by one of the  operations $O_9$ and $O_{10}$ listed below. 
\medskip
 \\
{\bf Operation}  {$O_9$}.  The labeled tree $(T^{i+1},S^{i+1})$ is obtained from 
$(T^i,S^i)$ and $(F_3, J^3)$ by adding  the edge $ux$, 
where $u \in V (T^i)$, $x \in V(F_3)$ and  $sta_{T^i}(u) = sta_{F_3}(x) = C$.
\medskip
 \\
{\bf Operation}  {$O_{10}$}.  The labeled tree $(T^{i+1},S^{i+1})$ is obtained from 
 $(T^i,S^i)$ and $(F_4, J^4)$ (see  Figure \ref{fig:F1234})  by adding  the edge $ux$, 
where $u \in V (T^i)$, $x \in V(F_4)$,  $sta_{T^i}(u) = D$, and $sta_{F_4}(x) = C$.
\medskip

Note that once a vertex is assigned a status, 
this status remains unchanged as the labeled tree $(T, S)$ is recursively constructed. 
By the above definitions we see that if $(T,S) \in \mathcal{R}_{01,02}$, 
then $S_A(T) = S_B(T) = \emptyset$. Therefore it is naturally to consider a labeling 
$S$ as $S:V (T ) \rightarrow\{C, D\}$.   

 From Theorem \ref{main} we immediately obtain the following result. 
\begin{corollary}\label{0102cc}
A tree  $T$ is in $\mathcal{R}_{02,012}$ if and only if there is a labeling
$S:V (T ) \rightarrow\{C, D\}$ such that  $(T, S)$ is in $\mathscr{T}_{02,012} $.
Moreover, if $(T, S) \in \mathscr{T}_{02,012}$ then $S_C(T) = V^{02}(T)$ and   	$S_D(T) = V^{012}(T)$.	
\end{corollary}

As un immediate consequence of Corollary \ref{unilab} we obtain: 
\begin{corollary}\label{unilab2}
If $(T,S_1),(T,S_2)\in \mathscr{T}_{02,012}$ then $S_1 \equiv S_2$.
\end{corollary}

\begin{them}\label{45} \cite{ckpw} 
If $G$ is a connected graph of order $n \geq 3$, then $\gamma_R(G) \leq  4n/5$.
 The equality holds if and only if $G$ is $C_5$ or is obtained from $\frac{n}{5}P_5$
 by adding a connected subgraph on the set of centers of the components of $\frac{n}{5}P_5$.
\end{them}

As a consequence of Theorem \ref{45} and Corollary \ref{0102cc} we have: 
\begin{corollary}\label{45cor}
Let $G$ be a connected $n$-vertex graph with $n \geq 6$ and $\gamma_R(G) =  4n/5$. 
Then $G$ is in $\mathcal{R}_{02,012}$ and $V^{012}(G)$ consists of all leaves and all stems.
 Moreover, if $G$ is a tree, then 
$G$ has a $\mathscr{T}$-sequence $\tau: (G^1, S^1), \dots, (G^j, S^j)$, ($j \geq 1$), 
 such that  $(G^1, S^1) = (F_3, J^3)$ (see  Figure \ref{fig:F1234})
 and  if $j \geq 2$, then $(G^{i+1}, S^{i+1})$ can be obtained recursively 
from $(G^i, S^i)$  by operation $O_9$. 
\end{corollary}

A graph $G$ is said to be in class $UVR$ if $\gamma(G-v) = \gamma (G)$ for each $v \in V(G)$. 
Constructive characterizations of trees belonging to $UVR$ are given
 in \cite{sam1} by the present author, and independently  in \cite{hh1} by Haynes and Henning. 
We need the following result in \cite{sam1} (reformulated in  our present terminology).

\begin{them}\label{uvr} \cite{sam1}
A tree $T$ of order at least $5$ is in  $UVR$ if and only if there is a labeling 
$S:V (T ) \rightarrow\{C, D\}$ such that  $(T, S)$ is in $\mathscr{T}_{02,012} $.
Moreover, if $(T, S) \in \mathscr{T}_{02,012}$ then 
$S_C(T)$ and $S_D(T)$ are the sets of all $\gamma$-bad and all $\gamma$-good 
vertices of $T$, respectively.
\end{them}
We end with our main result in this subsection. 

\begin{theorem}\label{main02012}
For any tree $T$ the following are equivalent:
\medskip

 {\rm($A_4$)} \ $T$ is in $\mathcal{R}_{02,012}$,   \  \  \  \  \ {\rm($A_5$)} $T$ is in $\mathscr{T}_{02, 012}$, 
 \  \   \   \   {\rm($A_6$)} $T$ is in $UVR$.
\end{theorem}

\begin{proof}[\bf Proof.]
Corollary \ref{0102cc} and Theorem \ref{uvr} together imply the required result.
\end{proof}

\section{Open problems and questions}
We conclude the paper by listing some interesting problems and directions for further research.
Let first note that if $n \geq 3$ and $\mathtt{G}_{n,k}$ is not empty, then $k \leq 4n/5$ (Theorem \ref{45}).  

 An element of $\mathbb{RE}_{n,k}$ is said to be {\em isolated}, 
whenever it is  both  maximal and  minimal. 
 In other words, a graph $H \in \mathtt{G}_{n,k}$ is isolated in $\mathbb{RE}_{n,k}$ 
if and only if $H \in \mathcal{R}_{CEA}$ and 
for each $e \in E(H)$ at least  one of the following holds: 
(a) $H-e$ is not connected,  (b) $\gamma_R(H) \not = \gamma_R(H-e)$,  
(c) $H-e$ is not $\gamma_R$-excellent.

\begin{example}\label{CEA-trees} 
\begin{itemize} 
\item[(i)] All  $\gamma_R$-excellent graphs with  the Roman domination number equals to $2$
                  are  $\overline{K_2}$ and $K_n$, $n \geq 2$. If a graph $G \in \mathcal{R}_{CEA}$ 
									and $\gamma_R(G) = 2$, then $G$ is complete. $K_n$ is isolated in $\mathbb{RE}_{n,2}$, $n \geq 2$. 
\item[(ii)] \cite{rhv} $K_2$, $H_7$ and $H_{8}$ (see Fig. \ref{fig:Fig 1}) are the only trees in $\mathcal{R}_{CEA}$.
\item[(iii)] If $\mathbb{RE}_{n,k}$ has a tree $T$ as an isolated element, then 
                     either $(n,k) = (2,2)$ and $T =K_2$,   or $(n,k) = (9,7)$ and $T = H_{7}$, or $(n,k) = (10,8)$ and $T = H_{8}$. 							
\end{itemize}
\end{example}

\begin{itemize}
\item[$\bullet$] Find results on the isolated elements of $\mathbb{RE}_{n,k}$. 
\end{itemize}


\begin{itemize}
\item[$\bullet$] What is the maximum number of edges $m(\mathtt{G}_{n,k})$ of a graph in $\mathtt{G}_{n,k}$?
                               Note that (a) $m(\mathtt{G}_{n,2}) = n(n-1)/2$,
															(b) $m(\mathtt{G}_{n,3}) = n(n-1)/2 - \left\lceil n/2\right\rceil$. 
\end{itemize} 

\begin{itemize}
\item[$\bullet$] Find results on those minimal elements of $\mathbb{RE}_{n,k}$ that are not trees. 
\end{itemize}
\begin{example} \label{mincycle}
(a)  A cycle $C_n$ is a minimal element of   $\mathbb{RE}_{n,k}$ if and only if $n \equiv 0 \pmod 3$
 and $k = 2n/3$. (b) A graph $G$ obtained from the complete bipartite graph $K_{p,q}$, $p \geq q \geq 3$,  
by deleting an edge is a minimal element of   $\mathbb{RE}_{p+q,4}$.
\end{example}
  The height of a poset  is the maximal number of elements of a chain.
\begin{itemize}
\item[$\bullet$] 
                               Find the height of  $\mathbb{RE}_{n,k}$. 
\end{itemize}
\begin{example} \label{height}
\begin{itemize}
\item[(a)]  It is easy to check that any  longest chain in $\mathbb{RE}_{6,4}$ 
has as the first element $H_3$ (see  Fig \ref{fig:Fig 1})
 and as the last element one of the two $3$-regular $6$-vertex graphs. 
Therefore the height of  $\mathbb{RE}_{6,4}$ is $5$. 
\item[(b)] Let us consider the poset $\mathbb{RE}_{5r,4r}$, $r \geq 2$. 
                     All its minimal elements  are $\gamma_R$-excellent trees
										(by Theorem \ref{45} and Corollary \ref{45cor}), which are characterized in Corollary \ref{45cor}. 
										Moreover, the graph obtained from $rP_5$ by adding a complete graph 
										on the set of centers of the components of $rP_5$ is the largest element of $\mathbb{RE}_{5r,4r}$. 
										Therefore the height of $\mathbb{RE}_{5r,4r}$ is $(r-1)(r-2)/2 + 1$. 
\end{itemize}
\end{example}

\begin{itemize}
\item[$\bullet$] Find results on $\gamma_{\mathtt{Y}R}$-excellent graphs at least when $\mathtt{Y}$ is one of 
$\{-1,0,1\}$, $\{-1,1\}$ and $\{-1,1,2\}$. 
\end{itemize}

\end{document}